\documentclass[12pt,reqno]{amsart}
\usepackage{amsmath,amssymb,amsthm}
\usepackage{geometry}
\usepackage{hyperref}
\usepackage[english]{babel}
\usepackage{microtype}
\usepackage{enumitem}
\usepackage{graphicx}
\newtheorem{theorem}{Theorem}[section]
\newtheorem{lemma}[theorem]{Lemma}
\newtheorem{corollary}[theorem]{Corollary}
\newtheorem{conjecture}[theorem]{Conjecture}
\theoremstyle{remark}

\newtheorem{remark}{Remark}[section]

\numberwithin{equation}{section}

\begin{document}

\title[New Congruences Involving $p$-adic dual sequences]
{New Congruences Involving $p$-adic dual sequences}

\author[Y. Otmani]{Yassine Otmani}
\address{Department of Mathematics, USTHB, RECITS Laboratory,
P.O.\ Box 32, Bab Ezzouar, El Alia 16111, Algiers, Algeria}
\email{yotmani@usthb.dz}


\subjclass[2020]{Primary 11A07; Secondary 11B65, 05A10}

\keywords{Congruences; Legendre symbol; dual sequence; generalized central trinomial coefficients; Catalan–Larcombe–French numbers; Legendre polynomials}

\begin{abstract}Let $(a_n)_{n\geqslant 0}$ be a sequence of integers. Its dual sequence
$(a_n^*)_{n\geqslant 0}$ is defined by
\begin{equation*}
  a_n^* := \sum_{k=0}^{n} \binom{n}{k}(-1)^k a_k.
\end{equation*}
 Let $p>3$ be a prime. In this paper we mainly investigate congruences modulo $p^2$ involving central binomial coefficients and $p$-adic dual sequences. For example, we prove that for any sequence $(a_k)_{k\ge0}$ of $p$-adic
integers, 
\begin{align*}
\sum^{(p-1)/2}_{k=0}\binom{2k}{k}^2\frac{a_{2k}}{16^k}\equiv\left( \frac{-1}{p}\right) \sum_{k=0}^{p-1}\frac{\mathcal{P}_{k}}{16 ^{k}}a_{k}^*\pmod{p^2},
\end{align*}
where $(\mathcal{P}_n)_{n\ge0}$ are the Catalan--Larcombe--French numbers given by
\begin{equation*}
  \mathcal{P}_0=1,\quad \mathcal{P}_1=8, \quad n^2 \mathcal{P}_n = 8(3n^2-3n+1)\mathcal{P}_{n-1}-128(n-1)^2\mathcal{P}_{n-2}
  \quad (n\ge2).
\end{equation*}
 We also  establish a new formula for $\sum_{k=0}^{(p-1)/2}\binom{2k}{k}a_{2k}^*/4^k \pmod{p^2}$ and as a consequence we confirm some
conjectures of Z.-W. Sun \cite{Sun2014CANT} on the generalized central trinomial
coefficients $T_{2k}(b,c)$, i.e., the coefficient of $x^{2k}$ in
$(x^2+bx+c)^{2k}$, where $b,c$ are integers.
\end{abstract}

\maketitle


\section{Introduction}

Let $p$ be an odd prime, $n$ be an integer, and let $\left(\frac{\cdot}{p}\right)$ denote the Legendre symbol, defined by
\begin{align*}
\left(\frac{n}{p}\right):=\begin{cases}0 &\text{if $p$ divides $n$},\\
1 &\text{if $n$ is a quadratic residue modulo $p$}, \\
-1 &\text{ if $n$ is a quadratic nonresidue modulo $p$}. \end{cases}
\end{align*}

In 2003, Rodriguez-Villegas~\cite{RV2003} conjectured the following four
supercongruences modulo $p^2$: for any prime $p > 3$,
\begin{align}
  \sum_{k=0}^{p-1} \frac{\binom{2k}{k}^2}{16^k}
  &\equiv \left(\frac{-1}{p}\right) \pmod{p^2}, \label{eq:RV1} \\
  \sum_{k=0}^{p-1} \frac{\binom{3k}{k}\binom{2k}{k}}{27^k}
  &\equiv \left(\frac{p}{3}\right) \pmod{p^2}, \label{eq:RV2} \\
  \sum_{k=0}^{p-1} \frac{\binom{4k}{2k}\binom{2k}{k}}{64^k}
  &\equiv \left(\frac{-2}{p}\right) \pmod{p^2} \label{eq:RV3} 
\end{align}
and
\begin{align}
 \sum_{k=0}^{p-1} \frac{\binom{6k}{3k}\binom{3k}{k}}{432^k}
  &\equiv \left(\frac{-1}{p}\right) \pmod{p^2}. \label{eq:RV4}
\end{align}  

These were confirmed by Mortenson~\cite{M1,M2} via the $p$-adic gamma function and
the Gross--Koblitz formula for character sums. Elementary proofs of \eqref{eq:RV1}--\eqref{eq:RV4}
were later given in \cite{ZHSun2014}.

Recall that for a sequence $(a_n)_{n\geqslant 0}$ of integers, its \emph{dual sequence}
$(a_n^*)_{n\geqslant 0}$ is defined by
\begin{equation}\label{eq:dual}
  a_n^* := \sum_{k=0}^{n} \binom{n}{k}(-1)^k a_k \qquad (n=0,1,2,\ldots).
\end{equation}
It is well known that $(a_n^*)^* = a_n$ for all $n \geqslant 0$
(see, e.g., \cite[p.~192]{GKP}).
Z.-H.~Sun~\cite{ZHSun2011} proved that for any odd prime $p$ and any sequence
$(a_k)_{k\geqslant 0}$ of $p$-adic integers,
\begin{equation}\label{eq:ZHSun}
  \sum_{k=0}^{(p-1)/2} \frac{\binom{2k}{k}^2}{16^k}
  \left(a_k - \left(\frac{-1}{p}\right) a_k^*\right) \equiv 0 \pmod{p^2}.
\end{equation}

In \cite{Sun2013FFA}, Z.-W.~Sun established the following analogues of \eqref{eq:ZHSun}
for the three remaining Rodriguez-Villegas weights.

\begin{theorem}[{Z.-W.~Sun \cite[Theorem~1.4]{Sun2013FFA}}]\label{thm:Sun14}
Let $p > 3$ be a prime and let $(a_k)_{k\geqslant 0}$ be any sequence of $p$-adic
integers. Then
\begin{align}
  \sum_{k=0}^{p-1} \frac{\binom{3k}{k}\binom{2k}{k}}{27^k}\, a_k
  &\equiv \left(\frac{p}{3}\right)
    \sum_{k=0}^{p-1} \frac{\binom{3k}{k}\binom{2k}{k}}{27^k}\, a_k^*
  \pmod{p^2}, \label{eq:T14-1} \\[4pt]
  \sum_{k=0}^{p-1} \frac{\binom{4k}{2k}\binom{2k}{k}}{64^k}\, a_k
  &\equiv \left(\frac{-2}{p}\right)
    \sum_{k=0}^{p-1} \frac{\binom{4k}{2k}\binom{2k}{k}}{64^k}\, a_k^*
  \pmod{p^2}, \label{eq:T14-2} \\[4pt]
  \sum_{k=0}^{p-1} \frac{\binom{6k}{3k}\binom{3k}{k}}{432^k}\, a_k
  &\equiv \left(\frac{-1}{p}\right)
    \sum_{k=0}^{p-1} \frac{\binom{6k}{3k}\binom{3k}{k}}{432^k}\, a_k^*
  \pmod{p^2}. \label{eq:T14-3}
\end{align}
\end{theorem}
A further dual-sequence supercongruence, this time for a single central binomial
coefficient rather than a product, was established in \cite{SunNan2015}. There Sun
showed that if $p$ is an odd prime, $(a_k)_{k\geqslant 0}$ is any sequence of $p$-adic
integers, and $m$ is an integer with $p \nmid m(m-4)$, then
\begin{equation}\label{eq:SunNan}
  \sum_{k=0}^{p-1}\frac{\binom{2k}{k}a_k^*}{(4-m)^k}
  \equiv \left(\frac{m(m-4)}{p}\right)
  \sum_{k=0}^{p-1}\frac{\binom{2k}{k}a_k}{m^k} \pmod{p}.
\end{equation}

In \cite{ZHSun2014}, Z.-H.~Sun extended \eqref{eq:ZHSun} to a more general framework
via generalized Legendre polynomials $P_n(a,x) = \sum_{k=0}^n
\binom{a}{k}\binom{-1-a}{k}\bigl(\frac{1-x}{2}\bigr)^k$. Specifically,
he showed~\cite[(1.10)]{ZHSun2014} that for any odd prime $p$, any
$p$-integral $\alpha \in \mathbb{Q}$, and any sequence $f(0),\ldots,f(p-1)$ of
$p$-adic integers,
\begin{equation}\label{eq:ZHSun2014}
  \sum_{k=0}^{p-1} \binom{\alpha}{k}\binom{-1-\alpha}{k}
  \!\left((-1)^{\langle\alpha\rangle_p} f(k) -
  \sum_{m=0}^{k}\binom{k}{m}(-1)^m f(m)\right) \equiv 0 \pmod{p^2},
\end{equation}
where $\langle\alpha\rangle_p$ is the least nonneg\-ative residue of $\alpha$ modulo
$p$. Taking $\alpha = -\frac{1}{2}$ recovers \eqref{eq:ZHSun}, and taking $\alpha =
-\frac{1}{3},-\frac{1}{4},-\frac{1}{6}$ yields the analogues for the weights in
\eqref{eq:RV2}--\eqref{eq:RV4}.

Supercongruences of the type \eqref{eq:ZHSun}--\eqref{eq:T14-3} have since been investigated from several directions. Z.-W.~Sun~\cite{Sun2011SC} proved that
\[
  \sum_{k=0}^{(p-1)/2}\frac{\binom{2k}{k}^2}{16^k}
  \equiv \left(\frac{-1}{p}\right) + p^2\,E_{p-3} \pmod{p^3},
\]
where $E_{p-3}$ is the $(p-3)$-th Euler number. Extensions to sums involving products of three binomial
coefficients were obtained by Sun~\cite{ZH.Sun2023, ZH.Sun2026}. On
the $q$-analogue side, a $q$-extension of \eqref{eq:ZHSun2014} was
established by Ni and Pan \cite[Theorem~1.1, (1.14)]{NiPan2020}, who proved that for any odd positive integer
$n \geqslant 3$ with $\gcd(n,d)=1$, any integer $r$, and any polynomials
$f_0(q),\ldots,f_{n-1}(q)\in\mathbb{Z}[q]$,
\begin{align}\label{eq:NiPan}
  q^{d\binom{a+1}{2}+\frac{(ad+r)(n-1-2a)}{2}}
  \sum_{k=0}^{n-1}\frac{(q^r;q^d)_k(q^{d-r};q^d)_k}{(q^d;q^d)_k^2}\,
  q^{dk}f_k(q^d)\nonumber \\
  \equiv (-1)^a
  \sum_{k=0}^{n-1}\frac{(q^r;q^d)_k(q^{d-r};q^d)_k}{(q^d;q^d)_k^2}\,
  q^{dk}\hat{f}_k(q^d)
  \pmod{\Phi_n(q)^2},
\end{align}
where $a=\langle -r/d\rangle_n$, $\Phi_n(q)$ is the $n$-th cyclotomic polynomial,
$(x;q)_k=(1-x)(1-xq)\cdots(1-xq^{k-1})$ is the $q$-shifted factorial, and
$\hat{f}_k(q) = \sum_{j=0}^k(-1)^j
q^{\binom{j+1}{2}}\genfrac{[}{]}{0pt}{}{k}{j}_q f_j(q)$.
Setting $q=1$ and $n=p$ in \eqref{eq:NiPan} recovers \eqref{eq:ZHSun2014}. For more background on $q$-analogues of
supercongruences concerning \eqref{eq:RV1}--\eqref{eq:RV4} we refer the reader to \cite{Guo2014,Gou2017}.

The \emph{Catalan-Larcombe-French numbers} $(\mathcal{P}_n)_{n\geqslant 0}$ are defined by
\begin{equation}\label{eq:CLF}
 \mathcal{P}_n = \sum_{k=0}^{n}\frac{\binom{2k}{k}^2\binom{2n-2k}{n-k}^2}{\binom{n}{k}}
  \qquad (n = 0,1,2,\ldots),
\end{equation}
giving $\mathcal{P}_0=1$, $\mathcal{P}_1=8$, $\mathcal{P}_2=80$, $\mathcal{P}_3=896$, $\mathcal{P}_4=10816$, \ldots\
(see \cite{LF2000}). They satisfy the recurrence
\begin{equation}\label{eq:CLFrec}
  n^2 \mathcal{P}_{n} = 8(3n^2-3n+1)\mathcal{P}_{n-1}-128(n-1)^2\mathcal{P}_{n-2} \quad (n\geqslant 2).
\end{equation}
 Supercongruences for the Catalan-Larcombe-French sequence were
initiated by Jarvis and Verrill~\cite{JV2010}, who proved that for any prime $p > 3$,
$\mathcal{P}_{p-1} \equiv 8^{p-1} \pmod{p^2}$.
Z.-H.~Sun~\cite{ZHSun2015} subsequently established congruences for
$\sum_{k=0}^{p-1}\binom{2k}{k}\mathcal{P}_k/m^k \pmod{p}$ for several values of $m$ and posed further conjectures modulo $p^2$; additional results appear in Ji and Sun~\cite{JiSun2015} and Mao \cite{mao2015}.

The \emph{generalized central trinomial coefficient} $T_n(b,c)$ is defined as the
coefficient of $x^n$ in the expansion of $(x^2+bx+c)^n$, i.e.,
\begin{equation}\label{eq:Tnbc}
  T_n(b,c) = \sum_{j=0}^{\lfloor n/2\rfloor}
  \binom{n}{2j}\binom{2j}{j}\,b^{n-2j}c^j \qquad (b,c\in\mathbb{Z}).
\end{equation}
The standard central trinomial coefficient is $T_n = T_n(1,1)$, the central binomial coefficient is $\binom{2n}{n} = T_n(2,1)$, while the Delannoy number is $D_n = T_n(3,2)$.\\
For $n$ be a nonnegative integer and $x$ be an indeterminate, the Legendre polynomial $P_n(x)$ is defined by
\begin{align*}
P_n(x):=\sum_{k=0}^n\binom{n}{k}\binom{n+k}{k}\left(\frac{x-1}{2}\right)^k.
\end{align*}
It is well known that $T_n(b,c)$ is closely related to the Legendre polynomial $P_n(x)$ via the following formula (see \cite{noe2006,Sun2014CANT}):
\begin{equation*}
T_n(b,c) = (\sqrt{d})^n P_n\left(\frac{b}{\sqrt{d}}\right),
\end{equation*} 
where $d=b^2-4c\neq0$.

 Z.-W.~Sun~\cite{Sun2014CANT} investigated sums of the form
$\sum_{k=0}^{p-1}\binom{2k}{k}T_{k}(b,c)/m^k$, $\sum_{k=0}^{p-1}\binom{2k}{k}T_{2k}(b,c)/m^k$, $\sum_{k=0}^{p-1}\binom{2k}{k}^2T_{k}(b,c)/m^k$ and $\sum_{k=0}^{p-1}\binom{2k}{k}^2T_{2k}(b,c)/m^k$ modulo $p$ and $p^2$ and posed many conjectures involving the above forms. For example, Z.-W.~Sun \cite[Conjectures~2.2]{Sun2014CANT} conjectured that: Let $p > 3$ be a prime. Then
\begin{align}\label{suncnj2.2}
&\sum_{k=0}^{(p-1)/2}\binom{2k}{k}\dfrac{T_{2k}(2,3)}{16^k}\equiv \sum_{k=0}^{(p-1)/2}\binom{2k}{k}\dfrac{T_{2k}(4,-3)}{16^k}\nonumber \\
&\equiv \begin{cases}
( \frac{-1}{p})\left( \frac{x}{3}\right)(2x-\frac{2p}{x}) \pmod p^2 &\text{ if }p=x^2+3y^2,\\
0 \pmod p & \text{ if }p\equiv 2\pmod 3.
\end{cases}
 \end{align}

Note that \cite[Conjectures~2.1]{Sun2014CANT} were confirmed partially in~\cite{WangSun2020} and \cite{wang2026}. For other congruences involving generalized central trinomial coefficients,  see \cite{sunsci2014,otmani2026, otmani2026+} and references therein.

Motivated by the above, we establish the following general theorem on congruences involving a product of two $p$-adic sequences.
\begin{theorem}\label{theorem1} Let $p>3$ be a prime. Let $(a_k)_{k\geqslant 0}$ and $(b_k)_{k\geqslant 0}$ be any sequences of $p$-adic integers. Then
\begin{equation}\label{th1c1}
\sum^{(p-1)/2}_{k=0}\binom{2k}{k}^2\frac{a_{k}b_k}{16^k}\equiv\left( \frac{-1}{p}\right) \sum_{k=0}^{(p-1)/2}\sum_{i=0}^{(p-1)/2}\frac{\binom{2k}{k}^2\binom{2i }{i}^2}{\binom{i+k}{k}16 ^{k+i}}a_{k}^*b_i^*\pmod{p^2}. 
\end{equation}
Moreover,
\begin{align}\label{congCLFS}
\sum^{(p-1)/2}_{k=0}\binom{2k}{k}^2\frac{a_{2k}}{16^k}\equiv\left( \frac{-1}{p}\right) \sum_{k=0}^{p-1}\frac{\mathcal{P}_{k}}{16 ^{k}}a_{k}^*\pmod{p^2}. 
\end{align}
\end{theorem}

\begin{remark}
Taking $b_k=1$ for all $k$ in \eqref{th1c1}, we immediately get \eqref{eq:ZHSun}. Thus \eqref{th1c1} may be viewed as a $(b_k)_{k\ge0}$-parametrized extension of Sun's congruence \eqref{eq:ZHSun}. Other choices of $(b_k)_{k\ge0}$ produce new congruences similar in form to Sun's result; see Section \ref{section4}.
\end{remark}

\begin{remark}
For nonnegative integers $n,m$, the sequence $\binom{2n}{n}\binom{2m}{m}/\binom{m+n}{n}$ are integers, and they are called the Super Catalan numbers (see \cite{gessel}).
\end{remark}

Our second goal in this paper is to establish the following congruence.
\begin{theorem}\label{theorem2}  Let $p>3$ be a prime. Let $(a_k)_{k\geqslant 0}$ be any sequence of $p$-adic integers. Then
\begin{equation}\label{theorem2cong}
\sum_{k=0}^{p-1}\binom{2k}{k}\dfrac{a_{k}^*}{4^k}\equiv \dfrac{p}{4^{p-1}}\sum_{k=0}^{p-1}\frac{a_k}{2k+1}-p H_{\frac{p-1}{2}}a_{\frac{p-1}{2}}\pmod{p^2},
\end{equation}
where $H_k:=\sum_{i=1}^{k}1/i$ is the $k$-th harmonic number with the convention $H_0=0$. 
\end{theorem}
Furthermore, we establish the following result.
\begin{theorem}\label{theorem3}
Let $p>3$ be a prime. Let $(a_k)_{k\geqslant 0}$ be any sequence of $p$-adic integers. Then
\begin{equation}\label{th3cng}
\sum_{k=0}^{(p-1)/2}\binom{2k}{k}\dfrac{a_{2k}^*}{4^k}\equiv p \left( \frac{-1}{p}\right)  \sum_{k=0}^{p-1}\frac{2^k}{(2k+1)\binom{2k}{k}}a_k\pmod{p^2}.
\end{equation}
\end{theorem}
\begin{remark}
Together with some further identities, Theorem \ref{theorem3} yields a
key congruence, Theorem \ref{thm5.6}, confirming \eqref{suncnj2.2}
modulo $p$, and confirming the remaining parts of Conjecture 2.2, together
with Conjecture 2.3, completely, as posed by Z.-W. Sun
\cite{Sun2014CANT}.
\end{remark}

The rest of this paper is organized as follows. In Section \ref{section2}, we prove our main theorems. In Section \ref{section3}, by applying \eqref{congCLFS} and Theorem \ref{theorem2}, we establish some new congruences involving the Catalan--Larcombe--French sequence and the generalized central trinomial coefficients. In Section \ref{section4}, we specialize the $p$-adic sequence $(b_k)_{k\ge0}$ in the congruence \eqref{th1c1} to obtain some new congruences involving other sequences $(a_k)_{k\ge0}$. In Section \ref{section5}, with the help of Theorem \ref{theorem3}, we establish some congruences involving $T_{2k}(b,c)$, and as a consequence we confirm some conjectures of Sun \cite{Sun2014CANT}. Our proofs make use of several combinatorial identities derived by different methods, including the symbolic summation package \texttt{Sigma} (Schneider \cite{sch2007}), Zeilberger's algorithm \cite{PWZ}, and Riordan array transformation \cite{spr1994}.

\section{Proofs of Mains Theorems}\label{section2}
We first collect some lemmas needed in the proofs of our main results.

\begin{lemma}\label{lemma2.1}  Let $n$ and $i$ be nonnegative integers. Then
\begin{align}\label{identity2k}
\sum_{k=0}^{n}\binom{n+k}{2k}\binom{2k}{i}(-4)^k &=(-1)^{n}(2n+1)2^i\dfrac{\binom{2n}{i}\binom{2n+1+i}{i}}{(2i+1)\binom{2i}{i}},\\
\label{lemma2.2+}
\sum_{k=i}^{n}\binom{2k}{k}\binom{k}{i}\dfrac{1}{4^k}&=\frac{n+1}{2^{2n+1}(2i+1)}\binom{n}{i}\binom{2(n+1)}{n+1}.
\end{align}
\end{lemma}
\begin{proof}We prove only \eqref{identity2k}, since \eqref{lemma2.2+} can be established by a similar argument. Let 
\begin{align*}
z_i:= \sum_{k=0}^{n}\binom{n+k}{2k}\binom{2k}{i}(-4)^k.
\end{align*}
By Zeilberger's algorithm \cite{PWZ}, the sequence $(z_i)_i$ satisfies the following recurrence relation
\begin{equation*}
2 (i-n) (1+i+n) z_{i}+(1+i) (3+2 i)z_{i+1}=0.
\end{equation*}
Solving the above recurrence relation, one obtains
\begin{align}\label{ziz0}
z_i=z_02^i\frac{\binom{2n}{i}\binom{2n+1+i}{i}}{(2i+1)\binom{2i}{i}}.
\end{align}
In view of \cite[(1.72)]{gould}, we have
\begin{align*}
z_0&=\sum_{k=0}^{n}\binom{n+k}{2k}(-4)^k=\sum_{k=0}^{n}\binom{2n-k}{k}(-4)^{n-k}\\
&=(-1)^n(2n+1).
\end{align*}
Hence, substituting this into \eqref{ziz0} yields the desired result.
\end{proof}

\begin{lemma}\label{lemma2.2} Let $n,i,j$ be nonnegative integers. Then

\begin{align*}
\sum_{k=\max\left( i,j\right) }^{n}\binom{n}{k}\binom{n+k}{k}\binom{k}{i}\binom{k}{j}\left( -1\right) ^k=\left( -1\right) ^{n}\binom{n}{i}\binom{n+j}{j}\binom{n+i}{i}\binom{n}{j}\binom{j+i}{j}^{-1}.
\end{align*}

\end{lemma}

\begin{proof}
For nonnegative integers $n,k$, and an indeterminate $x$, recall that 
\begin{align*}
\binom{x}{n}\binom{n}{k}=\binom{x}{k}\binom{x-k}{n-k}\quad \text{and}\quad \binom{-1-x}{k}=\binom{x+k}{k}(-1)^k.
\end{align*}
With the help of the above binomial identities we get
\begin{align*}
\sum_{k=\max\left( i,j\right) }^{n}\binom{n}{k}\binom{n+k}{k}\binom{k}{i}\binom{k}{j}\left( -1\right) ^k&=\binom{n}{i}\sum_{k=\max\left( i,j\right) }^{n}\binom{n-i}{k-i}\binom{n+k}{k}\binom{k}{j}\left( -1\right) ^k\\
&=\binom{n}{i}\binom{n+j}{j}\left( -1\right) ^{j}\sum_{k=\max\left( i,j\right) }^{n}\binom{n-i}{k-i}\binom{-n-1-j}{k-j}\\
&=\binom{n}{i}\binom{n+j}{j}\left( -1\right) ^{j}\sum_{k=\max\left( i,j\right) }^{n}\binom{n-i}{k-i}\binom{n+k}{k-j}\left( -1\right) ^{k+j}\\
&=\binom{n}{i}\binom{n+j}{j}\left( -1\right) ^{j}\sum_{k=0 }^{n}\binom{n-i}{k}\binom{n+i+k}{n+j}\left( -1\right) ^{k+i+j}.
\end{align*}
Applying \cite[(3.47)]{gould}, we obtain
\begin{align*}
\sum_{k=\max\left( i,j\right) }^{n}\binom{n}{k}\binom{n+k}{k}\binom{k}{i}\binom{k}{j}\left( -1\right) ^k&=\left( -1\right) ^{n}\binom{n}{i}\binom{n+j}{j}\binom{n+i}{i+j}\\
&=\left( -1\right) ^{n}\binom{n}{i}\binom{n+j}{j}\binom{n+i}{i}\binom{n}{j}\binom{j+i}{j}^{-1}.
\end{align*}
As desired.
\end{proof}

\begin{lemma}\label{lemma2.4}
Let $p$ be an odd prime and $k$ be a nonnegative integer.  If $k\in \{ 0,1,2,\ldots,p-1\}$, then
\begin{align}
\binom{p-1}{k}&\equiv (-1)^k(1-pH_k)\pmod{p^2}, \label{eq:binom1}\\
\binom{p+k}{k}&\equiv 1+pH_k \pmod{p^2}. \label{eq:binom2}
\end{align}
\end{lemma}
\begin{proof}
By definition of the binomial coefficient,
\begin{align*}
\binom{p-1}{k} &= \prod_{j=1}^{k}\frac{p-j}{j} 
= (-1)^k\prod_{j=1}^{k}\left(1-\frac{p}{j}\right)\\
&\equiv  (-1)^k(1-pH_k)\pmod{p^2}.
\end{align*}
Similarly,
\begin{align*}
\binom{p+k}{k}& = \prod_{j=1}^{k}\frac{p+j}{j} 
= \prod_{j=1}^{k}\left(1+\frac{p}{j}\right)\\
&\equiv 1+pH_k \pmod{p^2}.
\end{align*}
This completes the proof.
\end{proof}

\begin{lemma}\label{lemma2.3} Let $p=2n+1$ be a prime. Then
\begin{align*}
\binom{n+k}{2k}\equiv \binom{2k}{k}\dfrac{1}{(-16)^k}\pmod{p^2}.
\end{align*}
\end{lemma}
\begin{proof}
Note that
\begin{align*}
\binom{n+k}{2k}=\frac{1}{4^k(2k)!}\prod_{i=1}^{k}\left( p^2-(2j-1)^2\right). 
\end{align*}
It follows that
\begin{align*}
\binom{n+k}{2k}\equiv \frac{1}{4^k(2k)!}\prod_{i=1}^{k}\left(-(2j-1)^2\right)
=\binom{2k}{k}\frac{1}{(-16)^k}\pmod{p^2}.
\end{align*}
As desired.
\end{proof}

\begin{proof}[Proof of Theorem \ref{theorem1}]
 Let $p=2n+1$. By Lemma \ref{lemma2.3}, we obtain 
\begin{align*}
\sum^{n}_{k=0}\binom{2k}{k}^2\frac{a_{k}b_k}{16^k}
&\equiv \sum^{n}_{k=0}\binom{n}{k}\binom{n+k}{k}\left( -1\right) ^{k}a_{k}b_k\pmod{p^2}.
\end{align*}
In view of \eqref{eq:dual}, we have
\begin{align*}
\sum^{n}_{k=0}\binom{2k}{k}^2\frac{a_{k}b_k}{16^k}\equiv \sum_{j=0}^{n}\sum_{i=0}^{n}a^*_{j}b^*_i\left( -1\right)  ^{j+i}\sum_{k=\max\left( i,j\right) }^{n}\binom{n}{k}\binom{n+k}{k}\binom{k}{i}\binom{k}{j}(-1)^k\pmod{p^2}.
\end{align*}
Finally, applying Lemma \ref{lemma2.2} and then Lemma \ref{lemma2.3}
yields \eqref{th1c1}.

On the other hand, we have
\begin{align}
\sum^{n}_{k=0}\binom{2k}{k}^2\frac{a_{2k}}{16^k}& \equiv \sum^{n}_{k=0}\binom{2k}{k}^2\frac{1}{16^k}\sum_{i=0}^{2k}\binom{2k}{i}(-1)^ia_i^* \pmod{p^2} \nonumber \\
&=\sum_{i=0}^{p-1}(-1)^ia_i^*\sum^{(p-1)/2}_{k=\lfloor i/2\rfloor}\binom{2k}{k}^2\binom{2k}{i}\frac{1}{16^k}\pmod{p^2} \nonumber \\
&\equiv \sum_{i=0}^{p-1}(-1)^ia_i^*\sum^{n}_{k=\lfloor i/2\rfloor}\binom{n}{k}\binom{n+k}{k}\binom{2k}{i}(-1)^k \pmod{p^2} \label{cng25} .
\end{align}
By the Vandermonde identity \cite[(3.1)]{gould}, we have
\begin{align*}
\sum^{n}_{k=\lfloor i/2\rfloor}\binom{n}{k}\binom{n+k}{k}\binom{2k}{i}(-1)^k=\sum_{j=0}^{i} \sum^{n}_{k=\lfloor i/2\rfloor}\binom{n}{k}\binom{n+k}{k}\binom{k}{j}\binom{k}{i-j}(-1)^k.
\end{align*}
Applying Lemmas \ref{lemma2.2} and \ref{lemma2.3} and in view of \eqref{eq:CLF} we deduce that
\begin{align*}
\sum^{n}_{k=\lfloor i/2\rfloor}\binom{n}{k}\binom{n+k}{k}\binom{2k}{i}(-1)^k\equiv\left( \frac{-1}{p}\right) \frac{(-1)^i}{16^i} \mathcal{P}_i \pmod{p^2}.
\end{align*}
Substituting this into \eqref{cng25} gives  \eqref{congCLFS}. This completes the proof of Theorem \ref{theorem1}.
\end{proof}

\begin{proof}[Proof of Theorem \ref{theorem2}]
Note that
\begin{align*}
\sum_{k=0}^{p-1}\binom{2k}{k}\dfrac{a_{k}^*}{4^k}=\sum_{i=0}^{p-1}(-1)^ka_k\sum_{k=i}^{p-1}\binom{2k}{k}\binom{k}{i}\dfrac{1}{4^k}.
\end{align*}
Applying \eqref{lemma2.2+}, we get
\begin{align*}
\sum_{k=0}^{p-1}\binom{2k}{k}\dfrac{a_{k}^*}{4^k}=\dfrac{p}{2^{2(p-1)+1}}\binom{2p}{p}\sum_{i=0}^{p-1}(-1)^ka_k\frac{\binom{p-1}{k}}{2k+1}.
\end{align*}
Babbage's congruence \cite{babbage} states that for integers $a\ge b\ge0$,
\begin{equation*}
  \binom{ap}{bp}\equiv\binom{a}{b}\pmod{p^2}.
\end{equation*}
Finally, applying the above congruence with $a=2$, $b=1$, together with
Fermat's little theorem and \eqref{eq:binom1}, gives
\eqref{theorem2cong}.
\end{proof}
\begin{proof}[Proof of Theorem \ref{theorem3}] Observe that
\begin{align*}
\sum_{k=0}^{(p-1)/2}\binom{2k}{k}\dfrac{a_{2k}^*}{4^k}=\sum_{i=0}^{p-1}(-1)^ia_i\sum_{k=\lfloor i/2 \rfloor}^{(p-1)/2}\binom{2k}{k}\binom{2k}{i}\frac{1}{4^k}.
\end{align*}
By Lemma \ref{lemma2.3}, we obtain
\begin{align*}
\sum_{k=0}^{(p-1)/2}\binom{2k}{k}\dfrac{a_{2k}^*}{4^k}\equiv \sum_{i=0}^{p-1}(-1)^ia_i\sum_{k=\lfloor i/2 \rfloor}^{(p-1)/2}\binom{\frac{p-1}{2}+k}{2k}\binom{2k}{i}(-4)^k \pmod{p^2}.
\end{align*}
Using Lemma \ref{lemma2.1}, we get
\begin{align*}
\sum_{k=0}^{(p-1)/2}\binom{2k}{k}\dfrac{a_{2k}^*}{4^k}\equiv p\left( \frac{-1}{p}\right) \sum_{i=0}^{p-1}(-1)^ia_i\frac{2^i\binom{p-1}{i}\binom{p+i}{i}}{(2i+1)\binom{2i}{i}}\pmod{p^2}.
\end{align*}
Finally, combining \eqref{eq:binom1} and \eqref{eq:binom2} yields the desired result.
\end{proof}

\section{Some New Congruences involving Catalan-Larcombe-French Sequence and Generalized Central trinomial Coefficients }\label{section3}
We first treat congruence \eqref{congCLFS}, from which we derive several cases.
\begin{corollary} Let $p>3$ be a prime. Then
\begin{align}\label{ppcng}
\sum_{k=0}^{p-1}\frac{\mathcal{P}_{k}^2}{128 ^{k}} \equiv \left( \frac{-1}{p}\right)a(p)\pmod{p^2},
\end{align}
where $a(n)$ is defined by
\begin{align*}
q\prod_{n=1}^{\infty}(1-q^{2n})^4(1-q^{4n})^4=\sum_{n=1}^{\infty}a(n)q^n.
\end{align*}

\end{corollary}

\begin{proof}For a nonnegative integer $n$, it is well-known that (see \cite{Zagier2009})
\begin{align*}
\mathcal{P}_n=2^n\sum_{k=0}^{\lfloor n/2\rfloor}\binom{n}{2k}\binom{2k}{k}^24^{n-2k} .
\end{align*}
Define $a_n^*:=\mathcal{P}_n/8^n$. Recall that $\binom{2k}{k}\equiv 0\pmod{p}$ for $p/2<k<p$. Thus, it follows from \eqref{congCLFS} that
\begin{align*}
\sum_{k=0}^{p-1}\frac{\mathcal{P}_{k}^2}{128 ^{k}}\equiv\left( \frac{-1}{p}\right)  \sum^{p-1}_{k=0}\binom{2k}{k}^4\frac{1}{256^k}\pmod{p^2}.
\end{align*}
The result then follows from \cite[Theorem~5]{Ahlgren2000}.
\end{proof}
\begin{remark}
Numerical computation suggests that \eqref{ppcng} can be strengthened to
\begin{align}
\sum_{k=0}^{p-1}\frac{\mathcal{P}_{k}^2}{128 ^{k}}\equiv\left( \frac{-1}{p}\right)  \sum^{p-1}_{k=0}\binom{2k}{k}^4\frac{1}{256^k}\pmod{p^3}.
\end{align}
We leave this as an open problem.
\end{remark}

\begin{corollary}  Let $p>3$ be a prime. Then
\begin{align*}
p \sum_{k=0}^{p-1}\frac{\mathcal{P}_{k}}{(2k+1)\binom{2k}{k}  4^{k}}\equiv\begin{cases}4x^2-2p \pmod{p^2} & \text{ if }p=x^2+4y^2\equiv 1\pmod 4,
\\
0\pmod{p^2} & \text{ if }p\equiv 3 \pmod 4.
\end{cases}  
\end{align*}
\end{corollary}
\begin{proof}By \cite[(1.43)]{gould},
\begin{align*}
\sum_{k=0}^n\binom{n}{k}\frac{(-1)^k}{x-k}=\frac{(-1)^n}{(x-n)\binom{x}{n}}.
\end{align*}
Setting $x=-1/2$ in the above identity, and using $\binom{-\frac{1}{2}}{n}=\binom{2n}{n}\frac{1}{(-4)^n}$, gives
\begin{align}\label{4.3}
\sum_{k=0}^{n}\binom{n}{k}\frac{(-1)^k}{2k+1}=\frac{4^n}{(2n+1)\binom{2n}{n}}.
\end{align}
Define the $p$-adic sequence $a_k:=\frac{p}{2k+1}$, for $k\in \left\lbrace 0,1,\ldots,p-1 \right\rbrace $.
In view of \eqref{4.3} and \eqref{congCLFS}, we obtain
\begin{align*}
p\sum_{k=0}^{p-1}\frac{\mathcal{P}_{k}}{(2k+1)\binom{2k}{k}  4^{k}}\equiv p\left( \frac{-1}{p}\right) \sum^{p-1}_{k=0}\binom{2k}{k}^2\frac{1}{(4k+1)16^k} \pmod{p^2}.
\end{align*}
Combining Lemma \ref{lemma2.3} with the fact that
$\binom{2k}{k}\equiv 0\pmod{p}$ for $p/2<k<p$, we get
\begin{align}
p\sum_{k=0}^{p-1}\frac{\mathcal{P}_{k}}{(2k+1)\binom{2k}{k}  4^{k}}&\equiv p\left( \frac{-1}{p}\right) \sum^{(p-1)/2}_{k=0}\binom{2k}{k}^2\frac{1}{(4k+1)16^k} \pmod{p^2}\nonumber \\
&\equiv p\left( \frac{-1}{p}\right) \sum^{(p-1)/2}_{k=0}\binom{\frac{p-1}{2}}{k}\binom{\frac{p-1}{2}+k}{k}\frac{(-1)^k}{4k+1} \pmod{p^2}.\label{cor4.2eq}
\end{align}
As shown in the proof of \cite[Theorem 2.8]{ZHsun2018}, we have
\begin{align*}
p \sum^{\frac{p-1}{2}}_{k=0}\binom{\frac{p-1}{2}}{k}\binom{\frac{p-1}{2}+k}{k}\frac{(-1)^k}{4k+1}\equiv \begin{cases}4x^2-2p \pmod{p^2} & \text{ if }p=x^2+4y^2\equiv 1\pmod 4,
\\
0\pmod{p^2} & \text{ if }p\equiv 3 \pmod 4.
\end{cases}
\end{align*}
Finally, substituting the above congruence into \eqref{cor4.2eq} leads to the result.
\end{proof}

\begin{corollary} Let $p>3$ be a prime. Then
\begin{align*}
\sum_{k=0}^{p-1}\frac{\mathcal{P}_{k}H_k}{16^{k}}\equiv \left( \frac{-1}{p}\right)\left( -2q_p(2)+pq_p(2)^2\right)\pmod{p^2},
\end{align*}
where $q_p(2)=(2^{p-1}-1)/p$ is Fermat's quotient.
\end{corollary}
\begin{proof}
By \cite[(1.45)]{gould}, we have
\begin{align}\label{hnidentity}
\sum_{k=1}^{m}\binom{m}{k}\frac{(-1)^k}{k}=-H_m.
\end{align}
Define $a_k^{*}:=H_k$, for $k\in\{0,1,2,\ldots,p-1\}$. Applying \eqref{hnidentity} with \eqref{congCLFS} and in view of Lemma \ref{lemma2.3}, we obtain
\begin{align*}
\sum_{k=0}^{p-1}\frac{\mathcal{P}_{k}H_k}{16^{k}}&\equiv -\left( \frac{-1}{p}\right)\sum^{p-1}_{k=1}\binom{2k}{k}^2\frac{1}{2k16^k}  \pmod{p^2}\\
&\equiv -\left( \frac{-1}{p}\right)\sum^{\frac{p-1}{2}}_{k=1}\binom{\frac{p-1}{2}}{k}\binom{\frac{p-1}{2}+k}{k}\frac{(-1)^k}{2k}  \pmod{p^2}.
\end{align*}
Using the following identity (see \cite[(3.122)]{gould}),
\begin{align}\label{gould3.122}
\sum_{k=1}^m\binom{m}{k}\binom{m+k}{k}\frac{(-1)^k}{k}=-2H_m,
\end{align}
with $m=(p-1)/2$, we deduce that
\begin{align*}
\sum_{k=0}^{p-1}\frac{\mathcal{P}_{k}H_k}{16^{k}}&\equiv \left( \frac{-1}{p}\right)H_{(p-1)/2}  \pmod{p^2}.
\end{align*}
Thus, the result follows by applying $H_{(p-1)/2}\equiv -2q_p(2)+pq_p(2)^2\pmod{p^2} $ (see \cite{lehmer1938}).
\end{proof}

\begin{corollary}\label{corollary4.4} Let $p>3$ be a prime and $b,c$ be integers. If $p\nmid b$, then
\begin{align}\label{CFgcc}
\sum_{k=0}^{p-1}T_{k}(b,c)\frac{\mathcal{P}_{k}}{(16b) ^{k}}\equiv\left( \frac{-1}{p}\right)\sum^{p-1}_{k=0}\binom{2k}{k}^3\frac{c^{k}}{(4b)^{2k}} \pmod{p^2} .
\end{align}
\end{corollary}
\begin{proof}
In view of \eqref{eq:Tnbc}, one writes that
\begin{align*}
\frac{T_{n}\left( b,c\right)}{b^n}=\sum_{k=0}^{n}\binom{n}{k}a_k ,
\end{align*}
where
\begin{align*}
a_{k}=\begin{cases}
0, & \text{ if $k$ is odd},\\
 \binom{2k}{k}b^{-2k}c^k & \text{ if $k$ is even.}
\end{cases}
\end{align*}
Substituting the above into \eqref{congCLFS}, yields \eqref{CFgcc}.
\end{proof}

\begin{remark}
Taking $b=2$, $c=1$ in~\eqref{CFgcc} and applying 
\cite[Theorem~3]{ahl2001}, we deduce that
\begin{align*}
\sum_{k=0}^{p-1}\binom{2k}{k}\frac{\mathcal{P}_{k}}{32 ^{k}}\equiv \begin{cases}
4x^2-2p \pmod{p^2}& \text{ if }p=x^2+4y^2\equiv 1\pmod{4},  \\
0 \pmod{p^2} & \text{ if }p\equiv 3\pmod{4}.
\end{cases}
\end{align*}
One may also deduce similar results for other values of $b$, $c$ from Corollary~\ref{corollary4.4} by using the known evaluations of $\sum^{p-1}_{k=0}\binom{2k}{k}^3/m^{k}$ modulo $p^2$ in 
\cite{zhsun2013}.
\end{remark}

\begin{lemma}[{\cite[Theorem~3.7]{ArticleX}}]\label{lem:identity} Let $n$ be a nonnegative integer and $a,b,c$ be integers, with $a\neq 0$. We have
\begin{align*}
\sum_{k=0}^n\binom{n}{k}T_k(b,c)a^{n-k}=T_n(b+a,c).
\end{align*}
\end{lemma}

\begin{corollary}Let $p>3$ be a prime and $a,b,c$ be integers. If $p\nmid a$, then
\begin{align}\label{T2kT}
\sum^{p-1}_{k=0}\binom{2k}{k}^2\frac{T_{2k}(b,c) }{a^{2k}16^k}\equiv\left( \frac{-1}{p}\right) \sum_{k=0}^{p-1}\mathcal{P}_{k}\frac{T_k(b+a,c)}{a^{k}16^k}\pmod{p^2} 
\end{align}
\end{corollary}
\begin{proof}
\eqref{T2kT} follows by applying Lemma~\ref{lem:identity} to \eqref{congCLFS}.
\end{proof}
\begin{remark}Setting $b=2$ and $c=1$ in \eqref{T2kT}, we conclude that
\begin{align}\label{TkPk4k}
\sum_{k=0}^{p-1}\mathcal{P}_{k}\frac{T_k(2+a,1)}{a^{k}16^k}\equiv\left( \frac{-1}{p}\right)\sum^{p-1}_{k=0}\binom{2k}{k}^2\binom{4k}{2k}\frac{1}{a^{2k}16^k} \pmod{p^2} .
\end{align}
In particular, taking $a=\mp4$ in \eqref{TkPk4k} and in view of \cite{mort2005}, we get
\begin{align*}
\sum_{k=0}^{p-1}\binom{2k}{k}\frac{\mathcal{P}_{k}}{64 ^{k}}\equiv 
\sum_{k=0}^{p-1}\mathcal{P}_{k}\frac{T_k(6,1)}{64^k}\equiv
\begin{cases}
4x^2-2p \pmod{p^2}& \text{ if }p=x^2+2y^2\equiv 1,3\pmod{8},  \\
0 \pmod{p^2} & \text{ if }p\equiv 5,7\pmod{8}.
\end{cases}
\end{align*}
Other interesting congruences can be deduced by taking different values 
of $b$ and $c$ in \eqref{TkPk4k} with the help of 
\cite{zhsun2011,zhsun2013}.
\end{remark}

\begin{remark}
As mentioned above, Sun~\cite{Sun2014CANT} posed many conjectures on sums of 
the form $\sum_{k=0}^{p-1}\binom{2k}{k}^2T_{2k}(b,c)/m^k$ modulo $p^2$, 
where $p$ is an odd prime and $b$, $c$, $m$ are integers with $p\nmid m$. 
By \eqref{T2kT}, some of those sums admit a new representation in terms of 
$\sum_{k=0}^{p-1}\binom{2k}{k}\mathcal{P}_k/m^k$. For example, 
\cite[Conjecture~4.6]{Sun2014CANT} involves 
$\sum_{k=0}^{p-1}\binom{2k}{k}^2T_{2k}(6,1)/1024^k$, and by \eqref{T2kT} we have
\begin{align*}
\sum_{k=0}^{p-1}\binom{2k}{k}^2\frac{T_{2k}(6,1)}{1024^k}
\equiv \left( \frac{-1}{p}\right)\sum_{k=0}^{p-1}\binom{2k}{k}\frac{\mathcal{P}_k}{128^k}
\equiv \left( \frac{-1}{p}\right)\sum_{k=0}^{p-1}\frac{T_{k}(14,1)\mathcal{P}_k}{128^k}
\pmod{p^2}.
\end{align*}
Similar results for other values of $b$, $c$, and $m$ can be deduced 
from \eqref{T2kT} with the help of \cite[Conjectures~4.7--4.16]{Sun2014CANT}.
\end{remark}
\begin{lemma}\label{wanglemma} \cite[Lemma 2.1]{wangsun} Let $n$ be a nonnegative integer and $b,c$ be integers. We have
\begin{align*}
T_n(b,c^2)=\sum_{k=0}^{n}\binom{n}{k}\binom{2k}{k}(b+2c)^{n-k}(-c)^k.
\end{align*}
\end{lemma}
\begin{corollary}Let $p>3$ be a prime and $b,c$ be integers. If $p\nmid (b+2c)$, then
\begin{align}\label{T2kP}
\sum^{p-1}_{k=0}\binom{2k}{k}^2\frac{T_{2k}(b,c^2) }{(b+2c)^{2k}16^k}&\equiv\left( \frac{-1}{p}\right) \sum_{k=0}^{p-1}\binom{2k}{k}\frac{\mathcal{P}_{k}c^k}{(b+2c) ^{k}16^k}\pmod{p^2} ,\\
\label{T2kP1}
\sum_{k=0}^{p-1}\mathcal{P}_{k}\frac{ T_{k}(b,c^2)}{(b+2c) ^{k}16^k} &\equiv\left( \frac{-1}{p}\right)\sum^{p-1}_{k=0}\binom{2k}{k}^2\binom{4k}{2k}\frac{ c^{2k}}{(b+2c)^{2k}16^k} \pmod{p^2} .
\end{align}
\end{corollary}
\begin{proof}
\eqref{T2kP1} and \eqref{T2kP} are immediate consequences of Lemma~\ref{wanglemma} and \eqref{congCLFS}.
\end{proof}
In view of \eqref{T2kP} and \cite[Conjecture 3.1]{ZHSun2015}, 
we are led to pose the following conjecture.

\begin{conjecture}
Let $p>3$ be a prime, $n\in \{\pm 156816, \pm1584, \pm 784, \pm 144, \pm 48, \pm 9\}$ and $n \not\equiv 0, \pm 16 \pmod p$. Then
\begin{equation*}
  \sum^{p-1}_{k=0}\binom{2k}{k}^2\frac{T_{2k}(n,64) }{(n+16)^{2k}16^k} \equiv \left( \frac{-n(n+16)}{p}\right)\sum_{k=0}^{p-1}\binom{2k}{k}^2\binom{4k}{2k}\frac{1}{n^{2k}} \pmod{p^2}.
\end{equation*}
\end{conjecture}

We also establish the following result, which connects $T_{2n}(b,c)$ 
and $\mathcal{P}_{2n}$ modulo $p^2$, for any prime $p > 3$.
\begin{theorem}
Let $p>3$ be a prime and $b,c$ be integers. If $p\nmid b$, then 
\begin{equation}
\sum_{k=0}^{p-1}\binom{2k}{k}^2\frac{T_{2k}\left( b,c\right) }{\left( 4b\right) ^{2k}}\equiv \left( \frac{-1}{p}\right) \sum_{k=0}^{\left( p-1\right) /2}\binom{2k}{k}\frac{\mathcal{P}_{2k}c^k}{\left( 16b\right) ^{2k}}\pmod{p^2}.
\end{equation}
\end{theorem}
\begin{proof} Applying the fact that $\binom{2k}{k}\equiv 0\pmod p$ for $p/2<k<p$, and in view of \eqref{eq:Tnbc}, we have
\begin{align*}
\sum_{k=0}^{p-1}\binom{2k}{k}^2\frac{T_{2k}\left( b,c\right) }{\left( 4b\right) ^{2k}}&\equiv \sum_{k=0}^{\left( p-1\right) /2}\binom{2k}{k}^2\frac{1}{\left( 4b\right) ^{2k}}\sum_{i=0}^{k}\binom{2k}{2i}\binom{2i}{i}b^{2k-2i}c^i \pmod{p^2}  \\
&=\sum_{i=0}^{\left( p-1\right) /2}\binom{2i}{i}\frac{c^i}{b^{2i}}\sum_{k=i}^{\left( p-1\right) /2}\binom{2k}{k}^2\binom{2k}{2i}\frac{1}{16^k} \pmod{p^2}\\
&=\sum_{i=0}^{\left( p-1\right) /2}\binom{2i}{i}\frac{c^i}{b^{2i}}\sum_{k=i}^{\left( p-1\right) /2}\binom{2k}{k}^2\frac{1}{16^k}\sum_{m=0}^{2i}\binom{k}{m}\binom{k}{2i-m}\pmod{p^2} \\
&=\sum_{i=0}^{\left( p-1\right) /2}\binom{2i}{i}\frac{c^i}{b^{2i}}\sum_{m=0}^{2i}\sum_{k=i}^{\left( p-1\right) /2}\binom{2k}{k}^2\binom{k}{m}\binom{k}{2i-m}\frac{1}{16^k} \pmod{p^2}.
\end{align*}
Combining Lemmas \ref{lemma2.3} and \ref{lemma2.2} with \eqref{eq:CLF}, we deduce that
\begin{align*}
\sum_{k=0}^{p-1}\binom{2k}{k}^2\frac{T_{2k}\left( b,c\right) }{\left( 4b\right) ^{2k}}\equiv \left( \frac{-1}{p}\right) \sum_{i=0}^{\left( p-1\right) /2}\binom{2i}{i}\frac{\mathcal{P}_{2i}c^i}{\left( 16b\right) ^{2i}}\pmod{p^2}.
\end{align*}
This completes the proof of the theorem.
\end{proof}
We now turn our attention to some applications of Theorem \ref{theorem2}. First, we establish the following result.
\begin{corollary}Let $p>3$ be a prime and $a,b,c$ be integers. If $p\nmid a$, then
\begin{equation}\label{4.13}
\sum_{k=0}^{p-1}\binom{2k}{k}\dfrac{T_k(a+b,c)}{a^k4^k}\equiv \dfrac{p}{4^{p-1}}\sum_{k=0}^{p-1}\frac{T_k(b,c)(-1)^k(1-pH_k)}{(2k+1)a^k}\pmod{p^2}.
\end{equation}
\end{corollary}
\begin{proof}
\eqref{4.13} follows by applying Lemma \ref{lem:identity} and \eqref{theorem2cong}.
\end{proof}

\begin{remark}
Applying \eqref{eq:Tnbc} together with \eqref{theorem2cong} recovers the following special case:
\begin{equation*}
\sum_{k=0}^{p-1}\binom{2k}{k}\dfrac{T_k(b,c)}{(4b)^k}\equiv p\sum_{k=0}^{(p-1)/2}\binom{2k}{k}\frac{c^k}{(4k+1)b^{2k}}\pmod{p^2},
\end{equation*}
where $p\nmid b$. This result already appears as Theorem 1.2 in \cite{wangsun}.

In addition, by Lemma \ref{wanglemma} together with \eqref{theorem2cong}, we obtain the following case:
\begin{equation*}
\sum_{k=0}^{p-1}\binom{2k}{k}\dfrac{T_k(b,c^2)}{(b+2c)^k4^k}\equiv \dfrac{p}{4^{p-1}}\sum_{k=0}^{p-1}\frac{\binom{2k}{k}c^k}{(2k+1)(b+2c)^k}-p H_{\frac{p-1}{2}}\frac{\binom{p-1}{\frac{p-1}{2}}c^\frac{p-1}{2}}{(b+2c)^\frac{p-1}{2}}\pmod{p^2},
\end{equation*}
where $p\nmid c(b+2c)$. This congruence coincides with an intermediate step in the proof of \cite[Theorem 1.1]{wangsun}, where the right-hand side is further evaluated in closed form modulo $p^2$.

\end{remark}

\begin{corollary}\label{coro4.11}
Let $p>3$ be a prime and $b,c$ be integers. If $p\nmid (b+2c)$, then
\begin{equation*}\label{4.14}
 \dfrac{p}{4^{p-1}}\sum_{k=0}^{p-1}\frac{T_k(b,c^2)}{(2k+1)(b+2c)^k}\equiv \left( 1-2p q_p(2)\right)  \sum_{k=0}^{p-1}\binom{2k}{k}^2\dfrac{c^k}{(b+2c)^k4^k} \pmod{p^2}.
\end{equation*}
\end{corollary}
\begin{proof}
Define $a_k:= \frac{T_k(b,c^2)}{(b+2c)^k}$ for $k\in \{0,1,2,\ldots,p-1\}$. Applying Lemma \ref{wanglemma} and \eqref{theorem2cong}, we obtain
\begin{align}\label{4.15}
 \dfrac{p}{4^{p-1}}\sum_{k=0}^{p-1}\frac{T_k(b,c^2)}{(2k+1)(b+2c)^k} \equiv \sum_{k=0}^{p-1}\binom{2k}{k}^2\dfrac{c^k}{(b+2c)^k4^k} +p \left( \frac{b+2c}{p}\right) T_{\frac{p-1}{2}}(b,c^2)H_{\frac{p-1}{2}}\pmod{p^2}.
\end{align}
Recall that $\binom{\frac{p-1}{2}}{i}\equiv \frac{\binom{2i}{i}}{(-4)^i} \pmod p$ for $i\in\{0,1,2,\ldots,\frac{p-1}{2}\}$, and $\binom{2i}{i}\equiv 0\pmod p$ for $p/2<i<p$. These, together with Lemma \ref{wanglemma} and Lehmer's congruence $H_{\frac{p-1}{2}}\equiv -2q_p(2)\pmod p$, give
\begin{align*}
T_{\frac{p-1}{2}}(b,c^2)H_{\frac{p-1}{2}}\equiv -2q_p(2)\left( \frac{b+2c}{p}\right) \sum_{k=0}^{p-1}\binom{2k}{k}^2\dfrac{c^k}{(b+2c)^k4^k}\pmod p.
\end{align*} 
Substituting the above expression into \eqref{4.15} completes the proof.
\end{proof}
\begin{remark}
The reader may consult \cite{ZHSun2011} for some special cases of Corollary \ref{coro4.11}.
\end{remark}

\section{Some applications of the congruence \eqref{th1c1}}\label{section4}
In this section we apply the \texttt{Sigma} package to establish some new
combinatorial identities. Using these together with \eqref{th1c1}, we
derive some congruences similar to \eqref{eq:ZHSun}.

\begin{lemma}\label{lem:harmonic}
Let $n,i$ be positive integers with $n\ge i$. Then
\[
\sum_{k=1}^{n}\frac{\binom{n}{k}\binom{n+k}{k}(-1)^k}{k\binom{i+k}{k}}=-2H_n+H_i.
\]
\end{lemma}

\begin{proof}
We establish this identity (and other identities similar to it throughout this paper) by using the Mathematica package \texttt{Sigma} (Schneider \cite{sch2007}) as follows. We first insert the sum into \texttt{Sigma}:
\begin{align*}
\texttt{In[1]} &:= mySum=\texttt{SigmaSum[SigmaBinomial[n, k]*SigmaBinomial[n + k, k]}*\\
& \qquad\texttt{SigmaPower[-1, k]/(SigmaBinomial[i + k, k]*k), \{k, 1, n\}]}\\
\texttt{Out[1]} &=\sum_{k=1}^{n}\frac{\binom{n}{k}\binom{n+k}{k}(-1)^k}{k\binom{i+k}{k}}
\end{align*}
By creative telescoping (Zeilberger \cite{zei1991}), \texttt{Sigma} finds a recurrence for this sum:
\begin{align*}
\texttt{In[2]} &:= \texttt{rec = GenerateRecurrence[mySum,n][[1]]}\\
\texttt{Out[2]} &= \scalebox{0.85}{\text{-(1+n)(1-i+n)SUM[n]-(1+i)(3+2n)SUM[1+n]+(2+n)(2+i+n) SUM[2+n]==-2(3+2n)}}
\end{align*}
The correctness of this recurrence follows from a telescoping certificate (obtainable via \texttt{CreativeTelescoping[mySum]}), checkable directly by polynomial arithmetic on the summand. We then solve the recurrence for its d'Alembertian solutions --- nested sums and products built from objects already present in it:
\begin{align*}
\texttt{In[3]} &:= \texttt{recSol=SolveRecurrence[rec, SUM[n]]}\\
\texttt{Out[3]} &= \scalebox{0.85}{$\left\{\{0,1\},\left\{0,\frac{(2 i+1)\left(\underset{\iota _1=1}{\overset{n}{\sum}}\frac{(i-\iota _1+1)_{\iota _1}}{(i+2)_{\iota _1}}\right)}{i^2}+\frac{(i+1)\left(\underset{\iota _1=1}{\overset{n}{\sum}}\frac{(i-\iota _1) (i-\iota _1+1)_{\iota _1}}{\iota _1 (i+2)_{\iota _1}}\right)}{i^2}\right\},\left\{1,-2\left(\underset{\iota _1=1}{\overset{n}{\sum }}\frac{1}{\iota _1}\right)\right\}\right\}$}
\end{align*}
Since \texttt{mySum} itself satisfies the same recurrence, it must be a linear combination of these solutions; this combination is fixed by matching finitely many initial values:
\begin{align*}
\texttt{In[4]} &:= \texttt{FindLinearCombination[recSol, mySum, 2]}\\
\texttt{Out[4]} &= -2\left(\underset{\iota _1=1}{\overset{n}{\sum}}\frac{1}{\iota _1}\right)+\frac{2(2 i+1)\left(\underset{\iota _1=1}{\overset{n}{\sum }}\frac{(i-\iota _1+1)_{\iota _1}}{(i+2)_{\iota _1}}\right)}{i (i+1)}+\frac{2\left(\underset{\iota _1=1}{\overset{n}{\sum}}\frac{(i-\iota _1)(i-\iota _1+1)_{\iota _1}}{\iota _1 (i+2)_{\iota_1}}\right)}{i}
\end{align*}
As both sides (\texttt{Out[1]} and \texttt{Out[4]}) agree for finitely many $n$ and satisfy the same recurrence, they agree for all $n\ge0$.

Since $n\ge i$, the sums with upper limit $n$ in \texttt{Out[4]} truncate to upper limit $i$, giving
\begin{align}\label{5.3}
\sum_{k=1}^{n}\frac{\binom{n}{k}\binom{n+k}{k}(-1)^k}{k\binom{i+k}{k}}=-2H_n+\frac{2(2i+1)}{i(i+1)}\sum_{k=1}^{i}\frac{\binom{i}{k}}{\binom{i+1+k}{k}}+\frac{2}{i}\sum_{k=1}^{i}\frac{(i-k)\binom{i}{k}}{k\binom{i+1+k}{k}}.
\end{align}

By \texttt{Sigma} again, we find the following two closed forms for the sums on the right-hand side of \eqref{5.3}:
\begin{align*}
\sum_{k=1}^{i}\frac{\binom{i}{k}}{\binom{i+1+k}{k}}=-1+\frac{(i+1)4^i}{(2i+1)\binom{2i}{i}}
\end{align*}
and
\begin{align*}
\sum_{k=1}^{i}\frac{(i-k)\binom{i}{k}}{k\binom{i+1+k}{k}}=\frac{2 i+1}{i+1}-\frac{4^i}{\binom{2i}{i}}+\frac{iH_i}{2}.
\end{align*}
Finally, substituting these two identities into \eqref{5.3} and simplifying gives the result. This proves Lemma \ref{lem:harmonic}.
\end{proof}
\begin{remark}
Although Lemma \ref{lem:harmonic} is stated for positive integers $i$, the identity in fact continues to hold at $i=0$: taking $i=0$ recovers identity \cite[(3.122)]{gould}, since $H_0=0$. Thus Lemma \ref{lem:harmonic} holds more generally for every nonnegative integer $i\le n$.
\end{remark}
\begin{theorem}Let $p>3$ be a prime. Let $(a_k)_{k\geqslant0}$ be any sequence of $p$-adic integers. Then
\begin{align*}
\sum_{k=0}^{(p-1)/2}\binom{2k}{k}^2\frac{H_k}{16^{k}}\left( a_k+\left(\frac{-1}{p}\right) a^*_k\right) \equiv \left( \frac{-1}{p}\right)2H_{(p-1)/2}\sum_{k=0}^{(p-1)/2}\binom{2k}{k}^2\frac{a^*_k }{16 ^{k}} \pmod{p^2}.
\end{align*}
\end{theorem}
\begin{proof}
Let $p=2n+1$. Define $b_k:=H_k$ for $k\in\{0,1,2,\ldots,n\}$.
It follows from \eqref{hnidentity} and \eqref{th1c1} that
\begin{align*}
-\sum_{k=0}^{n}\binom{2k}{k}^2\frac{H_ka_k}{16^{k}}&\equiv \left( \frac{-1}{p}\right) \sum_{k=0}^{n}\sum_{i=1}^{n}\frac{\binom{2k}{k}^2\binom{2i }{i}^2}{\binom{i+k}{k}16 ^{k+i}}\frac{a_{k}^*}{i}\pmod{p^2}
\nonumber
\\
&\equiv \left( \frac{-1}{p}\right)\sum_{k=0}^{n}\binom{2k}{k}^2 \frac{a_{k}^*}{16 ^{k}}\sum_{i=1}^{n}\frac{\binom{n}{i}\binom{n+i}{i}}{\binom{i+k}{k}}\frac{(-1)^i}{i}\pmod{p^2}.
\end{align*}
Finally, an application of Lemma \ref{lem:harmonic} yields the desired result.
\end{proof}

\begin{theorem}Let $p>3$ be a prime. Let $(a_k)_{k\geqslant 0}$ be any sequence of $p$-adic integers. Then 
\begin{equation}\label{th5.1.38}
\sum_{k=0}^{p-1}\frac{\binom{2k}{k}^2}{16 ^{k}}\left( \frac{a_k}{\left( k+1\right)} +4\left( \frac{-1}{p}\right)ka^*_k\right) \equiv 0 \pmod{p^2}.
\end{equation}
\end{theorem}
\begin{proof}Let $m$ be a nonnegative integer. We have (see \cite[(1.37)]{gould})
\begin{align*}
\sum_{k=0}^{m}\binom{m}{k}\frac{(-1)^k}{k+1}=\frac{1}{m+1}.
\end{align*}
Let $p=2n+1$. Define $b_k:=\frac{1}{k+1}$, for $k \in \{0,1,2,\ldots,n\}$.  In view of the above binomial identity and by \eqref{th1c1}, we obtain
\begin{align}
\sum_{k=0}^{p-1}\binom{2k}{k}^2\frac{a_k}{\left( k+1\right) 16^{k}}&\equiv \left( \frac{-1}{p}\right) \sum_{k=0}^{n}\sum_{i=0}^{n}\frac{\binom{2k}{k}^2\binom{2i }{i}^2}{\binom{i+k}{k}16 ^{k+i}}\frac{a_{k}^*}{i+1}\pmod{p^2}
\nonumber
\\
&\equiv \left( \frac{-1}{p}\right)\sum_{k=0}^{n}\binom{2k}{k}^2 \frac{a_{k}^*}{16 ^{k}}\sum_{i=0}^{n}\frac{\binom{n}{i}\binom{n+i}{i}}{\binom{i+k}{k}}\frac{(-1)^i}{i+1}\pmod{p^2}.\label{th5.1.39}
\end{align}
By \texttt{Sigma}, for a positive integer $m$ we obtain the following
identity:
\begin{align}\label{lemma5.1}
\sum_{k=0}^{m}\frac{\binom{m}{k}\binom{m+k}{k}\left( -1\right) ^{k}}{\left( k+1\right) \binom{i+k}{k}}=\frac{i}{m\left( m+1\right) }.
\end{align}
Taking $m=n$ in \eqref{lemma5.1} and substituting the resulting identity
into \eqref{th5.1.39} gives \eqref{th5.1.38}, which completes the proof.
\end{proof}

\begin{theorem}Let $p>3$ be a prime. Let $(a_k)_{k\geqslant0}$ be any sequence of $p$-adic integers. Then
\begin{equation}
\sum_{k=0}^{(p-1)/2}\frac{\binom{2k}{k}^2}{\left( 2k-1\right) 16^{k}}\left( a_k+2k\left( \frac{-1}{p}\right)a_k^*\right) \equiv p\sum_{k=0}^{(p-1)/2}\binom{2k}{k}\frac{a_k^*}{\left( 2k-1\right) 4 ^{k}}\pmod{p^2}.
\end{equation}
\end{theorem}
\begin{proof}Let $k$ be a nonnegative integer. Recall that
\begin{align*}
\binom{-\frac{1}{2}}{k}=\binom{2k}{k}\frac{1}{(-4)^k}\quad\text{and}\quad\binom{-1-x}{k}=\binom{x+k}{k}(-1)^k.
\end{align*}
In view of the above identities and by \cite[(4.1)]{gould}, one obtains
\begin{align}\label{5.5}
\sum_{k=0}^m\frac{\binom{2k}{k}}{4^k\binom{x+k}{k}}=\frac{2x}{2x-1}-\frac{\left( 2m+1\right) \binom{2m}{m}}{\left( 2x-1\right) \binom{x+m}{m}4^m}.
\end{align}
By \cite[(1.41)]{gould},
\begin{align*}
\sum_{k=0}^m\binom{m}{k}\frac{\left( -1\right) ^{k}}{2k-1}=-\frac{4^m}{\binom{2m}{m}}.
\end{align*}
 Applying \eqref{th1c1}, we have
\begin{align*}
-\sum_{k=0}^{(p-1)/2}\binom{2k}{k}^2\frac{a_k}{16^{k}(2k-1)}&\equiv 
\left( \frac{-1}{p}\right)\sum_{k=0}^{(p-1)/2}\binom{2k}{k}^2 \frac{a_{k}^*}{16 ^{k}}\sum_{i=0}^{(p-1)/2}\frac{\binom{2i}{i}}{\binom{i+k}{k}4^i}\pmod{p^2}.
\end{align*}
By \eqref{5.5}, we obtain
\begin{align*}
-\sum_{k=0}^{(p-1)/2}\binom{2k}{k}^2\frac{a_k}{16^{k}(2k-1)}\equiv & 
2\left( \frac{-1}{p}\right)\sum_{k=0}^{(p-1)/2}\binom{2k}{k}^2 \frac{a_{k}^*k}{16 ^{k}(2k-1)}\\
&-p\frac{\binom{p-1}{\frac{p-1}{2}}}{4^{(p-1)/2}} \left( \frac{-1}{p}\right)\sum_{k=0}^{(p-1)/2}\binom{2k}{k}^2 \frac{a_{k}^*}{16 ^{k}(2k-1)\binom{n+k}{k}} \pmod{p^2}.
\end{align*}
A direct computation shows that $\binom{n+k}{k}\equiv \binom{2k}{k}/4^k\pmod p$ for $k\in \{0,1,2,\ldots,\frac{p-1}{2} \}$. This, together with Morley's congruence \cite{Morley1895} and Fermat's little theorem, yields the desired result.
\end{proof}

\begin{theorem}Let $p>3$ be a prime. Let $(a_k)_{k\geqslant0}$ be any sequence of $p$-adic integers. Then
\begin{equation}
\sum_{k=0}^{(p-1)/2}\frac{\binom{2k}{k}}{ 4^{k}}\left(  p \frac{a_k}{( 2k+1)}-\left( \frac{-1}{p}\right)a_k^*\right)  \equiv  0\pmod{p^2}.
\end{equation}
\end{theorem}
\begin{proof}Let $p=2n+1$ and $k\in \{0,1,2,\ldots, n\}$. Define $b_k:=\frac{p}{2k+1}$. In view of \eqref{4.3} and \eqref{th1c1}, we have
\begin{align}
p\sum_{k=0}^{p-1}\binom{2k}{k}\frac{a_k}{4^{k}(2k+1)}&\equiv \left( \frac{-1}{p}\right) \sum_{k=0}^{n}\sum_{i=0}^{n}\frac{\binom{2k}{k}^2\binom{2i }{i}^2}{\binom{i+k}{k}(2i+1)16 ^{k+i}}a_{k}^*  \pmod{p^2}
\nonumber
\\
&= p\left( \frac{-1}{p}\right)\sum_{k=0}^{n}\binom{2k}{k}^2 \frac{a_{k}^*}{16 ^{k}}\sum_{i=0}^{n}\frac{\binom{n}{i}\binom{n+i}{i}(-1)^i}{\binom{i+k}{k}(2i+1)} \pmod{p^2}.\label{5.6}
\end{align}
By \texttt{Sigma}, we find the following identity holds, for $m \ge i $:
\begin{align*}
\sum_{k=0}^{m} \binom{m}{k} \binom{m+k}{k} \frac{(-1)^k}{(2k+1)\binom{i+k}{k}}
=\frac{4^i}{\binom{2i}{i}(2m+1)}.
\end{align*}
Substituting the above identity, for $m=n$, into \eqref{5.6} gives the desired result.
\end{proof}
\begin{theorem}Let $p>3$ be a prime. Let $(a_k)_{k\geqslant0}$ be any sequence of $p$-adic integers. Then
\begin{equation*}
\sum_{k=0}^{(p-1)/2}\frac{\binom{2k}{k}^2}{ 16^{k}}\left(  \frac{H_{k+1}}{k+1}a_k+4\left( \frac{-1}{p}\right)\left(4k+1+2kH_{(p-1)/2}-kH_k \right) a_k^*\right)  \equiv  0\pmod{p^2}.
\end{equation*}
\end{theorem}
\begin{proof}
By \texttt{Sigma}, the following identities hold:
\begin{align}\label{5.10}
\sum_{k=0}^{m}\binom{m}{k}\frac{(-1)^k}{(1+k)^2}=\frac{H_{m+1}}{m+1}
\end{align}
and
\begin{align}\label{5.11}
\sum_{k=0}^{m} \binom{m}{k} \binom{m+k}{k} \frac{(-1)^k}{(k+1)^2\binom{i+k}{k}}
&=\frac{-i+m+m^2}{m^2(m+1)^2}+\frac{2iH_m}{m(m+1)}-\frac{iH_i}{m(m+1)}.
\end{align}
Define $b_k:=\frac{H_{k+1}}{k+1}$. Applying \eqref{5.10} and \eqref{th1c1} with the help of Lemma \ref{lemma2.3}, we have
\begin{align*}
\sum_{k=0}^{(p-1)/2}\frac{\binom{2k}{k}^2}{ 16^{k}}\frac{H_{k+1}}{k+1}a_k\equiv \sum_{k=0}^{(p-1)/2}\frac{\binom{2k}{k}^2}{ 16^{k}}a_k^*\sum_{i=0}^{(p-1)/2} \binom{\frac{p-1}{2}}{i} \binom{\frac{p-1}{2}+i}{i} \frac{(-1)^i}{(i+1)^2\binom{i+k}{k}}\pmod{p^2}.
\end{align*}
Substituting \eqref{5.11}, for $m=(p-1)/2$, into the above congruence yields the desired result.
\end{proof}
\begin{remark}
We recall that the congruence \eqref{th1c1} in fact holds for an arbitrary $p$-adic sequence $(b_k)_{k\ge0}$ as well; in this section we have only specialized $(b_k)_{k\ge0}$ to obtain the theorems above. Other choices of $(b_k)_{k\ge0}$ would lead to further families of congruences; we leave it to the interested reader to explore this and other specializations.
\end{remark}
\begin{remark}
The theorems established in this section hold for an arbitrary sequence $(a_k)_{k\ge0}$ of $p$-adic integers, and we have not specialized $(a_k)_{k\ge0}$ to obtain any particular congruence. We invite the interested reader to explore such special cases; for instance, one could take $a_k=T_k(b,c)/b^k$.
\end{remark}
\section{Some congruences involving $T_{2k}(b,c)$}
\label{section5}
We first apply Theorem~\ref{theorem3} to the different formulas for the generalized central trinomial coefficients.
\begin{corollary}Let $p>3$ be a prime and $a,b,c$ be integers. If $p\nmid a$, then
\begin{equation}\label{5.1+}
\sum_{k=0}^{(p-1)/2}\binom{2k}{k}\dfrac{T_{2k}(b+a,c)}{a^{2k}4^k}\equiv p \left( \frac{-1}{p}\right)  \sum_{k=0}^{p-1}\frac{T_k(b,c)(-2)^k}{a^k(2k+1)\binom{2k}{k}}\pmod{p^2}.
\end{equation}
\end{corollary}
\begin{proof}
\eqref{5.1+} follows by applying Lemma \ref{lem:identity} and \eqref{th3cng}.
\end{proof}
\begin{corollary}Let $p>3$ be a prime and $b,c$ be integers. If $p\nmid b$, then
\begin{equation}\label{5.2+}
\sum_{k=0}^{(p-1)/2}\binom{2k}{k}\dfrac{T_{2k}(b,c)}{(4b^2)^k}\equiv p \left( \frac{-1}{p}\right)  \sum_{k=0}^{(p-1)/2}\binom{2k}{k}\frac{(4c)^{k}}{b^{2k}(4k+1)\binom{4k}{2k}}\pmod{p^2}.
\end{equation}
\end{corollary}
\begin{proof}
Applying \eqref{eq:Tnbc} together with \eqref{th3cng} yields \eqref{5.2+}.
\end{proof}
\begin{corollary}Let $p>3$ be a prime and $b,c$ be integers. If $p\nmid (b+2c)$, then
\begin{equation*}
 p  \sum_{k=0}^{p-1}\frac{T_k(b,c^2)2^k}{(b+2c)^k(2k+1)\binom{2k}{k}} \equiv \left( \frac{-1}{p}\right) \sum_{k=0}^{(p-1)/2}\binom{2k}{k}\binom{4k}{2k}\dfrac{c^{2k}}{(b+2c)^{2k}4^k}\pmod{p^2}.
\end{equation*}
\end{corollary}
\begin{proof}
Combining Lemma \ref{wanglemma} and \eqref{th3cng} gives the desired result.
\end{proof}
\begin{remark}Another application of Lemma \ref{wanglemma} together with \eqref{th3cng} gives
\begin{equation*}
\sum_{k=0}^{(p-1)/2}\binom{2k}{k}\dfrac{ T_{2k}(b,c^2)}{(b+2c)^{2k}4^k}\equiv p \left( \frac{-1}{p}\right)  \sum_{k=0}^{p-1}\frac{c^k2^k}{(b+2c)^k(2k+1)}\pmod{p^2},
\end{equation*}
where $p\nmid (b+2c)$. This result already appears in \cite[Theorem 1.1]{wang2026}.
\end{remark}
We now turn to a general theorem confirming conjectures of Sun, whose proof relies on a combinatorial identity obtained through Riordan array transformation.

Recall that a pair of formal power series $(g(t),h(t))$ with $g(0)\ne0$, $h(0)=0$, and $h'(0)\ne0$ is called a \emph{Riordan array} (cf.~\cite{spr1994}), and is identified with the infinite lower triangular matrix whose $(n,k)$-entry is
\[
d_{n,k} = [t^n]\, g(t)h(t)^k,
\]
where $[t^n]\,f(t)$ denotes the coefficient of $t^n$ in the formal power series $f(t)$.

A fundamental result on Riordan arrays, often called the \emph{fundamental theorem of Riordan arrays} (cf.~\cite[Theorem 1.1]{spr1994}), states that if
\(
f(t)=\sum_{k\ge 0} f_k t^k
\), then
\begin{equation*}
\sum_{k=0}^n d_{n,k} f_k = [t^n]\, g(t)\, f\!\big(h(t)\big).
\end{equation*}
Moreover, Sprugnoli \cite[(2.1))]{spr1994} investigated the above transformation, where he proved that: 
\begin{equation}\label{RAT}
\sum_k \binom{n+ak}{m+dk}f_k=[t^n]\frac{t^m}{(1-t)^{m+1}}f\left(\frac{t^{d-a}}{(1-t)^d}\right) \text{ if }d>a.
\end{equation}
We now use \eqref{RAT} to establish a new combinatorial identity involving generalized central trinomial coefficients.

\begin{lemma}\label{lemma5.4}
Let $n$ be a positive integer and $x$ be an indeterminate. Then
\begin{align*}
 \sum_{k=0}^{n-1}
        \binom{n+k}{2k+1}\frac{T_k(b,c)}{x^k}
    = \sum_{k=0}^{n-1}
        P_k\left(1 + \frac{b + 2\sqrt{c}}{2x} \right) 
        P_{n-1-k}\left(1 + \frac{b -2\sqrt{c}}{2x} \right) .
\end{align*}
\end{lemma}
\begin{proof}Recall that the generating function of generalized central trinomial coefficients is
\begin{align*}
   \sum_{k=0}^{\infty} T_k(b,c)t^k= \frac{1}{\sqrt{1 - 2bt + (b^2-4c)t^2}}.
\end{align*}
In addition, we have
\begin{align*}
   \sum_{k=0}^{\infty} P_k(x)t^k= \frac{1}{\sqrt{1 - 2t +t^2}}.
\end{align*}
Taking $a=2$, $d=1$, $m=1$ and $f_k=\frac{T_k(b,c)}{x^k}$ in \eqref{RAT}, we obtain
\begin{align*}
\sum_{k=0}^{n-1}
        \binom{n+k}{2k+1}\frac{T_k(b,c)}{x^k}
  & =[t^n]  \frac{t}{(1-t)^2} \frac{x}{\sqrt{x^2-\dfrac{2bxt}{(1-t)^2}+ \dfrac{(b^2-4c)t^2}{(1-t)^4}}}\\
 & =[t^n]\frac{xt}{\sqrt{x^2(1-t)^4- 2bxt(1-t)^2+ (b^2-4c)t^2}}\\
 &=[t^n]\frac{xt}{\sqrt{\bigl( x(1-t)^2 - bt\bigr)^2 - 4ct^2}}\\
 &=[t^n] \frac{xt}{\sqrt{(1-(1 + \frac{b+2\sqrt{c}}{2x}) t+t^2)(1-(1 + \frac{b - 2\sqrt{c}}{2x}) t+t^2)}}.
\end{align*}
Finally, extracting $[t^n]$ via the Cauchy product gives the desired result.
\end{proof}
\begin{remark}
We note that the combinatorial identity in Lemma~\ref{lemma5.4} does not appear to be directly accessible by automated methods such as the Sigma package or Zeilberger's algorithm.
\end{remark}
\begin{lemma}[{\cite[Theorem 3.1]{Sun2014CANT}}] \label{lemma5.5} For any nonnegative integer $n$ and indeterminates $x,y$, we have
\begin{align*}
P_n(x)P_n(y)=\sum_{k=0}^{n}\binom{n+k}{2k}\binom{2k}{k}\sum_{j=0}^{k}\binom{k+j}{2j}\binom{2j}{j}\left((x+1)(y-1) \right) ^{j}\left(x-y \right) ^{k-j}\frac{1}{2^{k+j}}.
\end{align*}
\end{lemma}

\begin{theorem}\label{thm5.6}
Let $p>3$ be a prime. Let $a,b,c$ be integers and let $d=b^2-4c$. If $p\nmid a$, then
\begin{equation}\label{5.4+}
\sum_{k=0}^{p-1}\binom{2k}{k}\dfrac{T_{2k}(b+a,c)}{a^{2k}4^k}\equiv \left( \frac{2a}{p}\right)\left( \sqrt{c}\right) ^{\frac{p-1}{2}}P_{\frac{p-1}{2}}\left( \frac{2ab+d}{4a\sqrt{c}}\right)  \pmod p.
\end{equation}
\end{theorem}
\begin{proof}Let $k\in\{0,1,2,\ldots,p-1\}$. It is easy to show that
\begin{align*}
\binom{p+k}{2k+1}\equiv p\frac{(-1)^k}{(2k+1)\binom{2k}{k}}\pmod{p^2}.
\end{align*}
This, with \eqref{5.1+} yields
\begin{equation*}
\sum_{k=0}^{(p-1)/2}\binom{2k}{k}\dfrac{T_{2k}(b+a,c)}{a^{2k}4^k}\equiv \left( \frac{-1}{p}\right)  \sum_{k=0}^{p-1}\binom{p+k}{2k+1}\frac{T_k(b,c)2^k}{a^k}\pmod{p^2}.
\end{equation*}
Applying Lemma \ref{lemma5.4}, for $n=p-1$, we have
\begin{equation*}
\sum_{k=0}^{(p-1)/2}\binom{2k}{k}\dfrac{T_{2k}(b+a,c)}{a^{2k}4^k}\equiv \left( \frac{-1}{p}\right)\sum_{k=0}^{p-1}
        P_k\left(1 + \frac{b + 2\sqrt{c}}{a} \right) 
        P_{p-1-k}\left(1 + \frac{b -2\sqrt{c}}{a} \right)\pmod{p^2}.
\end{equation*}
It is known that $P_{p-1-k}\equiv P_k \pmod p$ (see \cite[Lemma 6.2]{wahab1952}). Hence 
\begin{equation*}
\sum_{k=0}^{(p-1)/2}\binom{2k}{k}\dfrac{T_{2k}(b+a,c)}{a^{2k}4^k}\equiv \left( \frac{-1}{p}\right)\sum_{k=0}^{p-1}
        P_k\left(1 + \frac{b + 2\sqrt{c}}{a} \right) 
        P_{k}\left(1 + \frac{b -2\sqrt{c}}{a} \right)\pmod{p}.
\end{equation*}
By Lemma \ref{lemma5.5}, we obtain
\begin{align*}
&\sum_{k=0}^{(p-1)/2}\binom{2k}{k}\dfrac{T_{2k}(b+a,c)}{a^{2k}4^k}\\
&\equiv \left( \frac{-1}{p}\right)\sum_{k=0}^{p-1}\sum_{i=0}^{k}\binom{k+i}{2i}\binom{2i}{i}\sum_{j=0}^{i}\binom{i+j}{2j}\binom{2j}{j}\frac{\left(2a+b + 2\sqrt{c} \right) ^{j}\left(b-2\sqrt{c} \right) ^{j}\left(4\sqrt{c} \right) ^{i-j}}{(2a)^{i+j}} \pmod{p}\\
&=\sum_{i=0}^{p-1}\left( \frac{-1}{p}\right)\binom{2i}{i}\sum_{j=0}^{i}\binom{i+j}{2j}\binom{2j}{j}\frac{\left(2a+b + 2\sqrt{c} \right) ^{j}\left(b-2\sqrt{c} \right) ^{j}\left(4\sqrt{c} \right) ^{i-j}}{(2a)^{i+j}}\sum_{k=i}^{p-1}\binom{k+i}{2i}\pmod p.
\end{align*}
Recall that $\binom{2i}{i}\equiv 0 \pmod p$ if $p/2<i\leq p-1.$ Also note that
\begin{align*}
\sum_{k=i}^{p-1}\binom{k+i}{2i}=\sum_{k=0}^{p-1-i}\binom{k+2i}{2i}\equiv \sum_{k=0}^{p-1-i}\binom{p-1-2i}{k}(-1)^k\pmod p.
\end{align*}
Combining the above results, we have
\begin{align*}
&\sum_{k=0}^{(p-1)/2}\binom{2k}{k}\dfrac{T_{2k}(b+a,c)}{a^{2k}4^k}\\
&\equiv \left( \frac{-1}{p}\right)\binom{p-1}{\frac{p-1}{2}}\sum_{j=0}^{(p-1)/2}\binom{\frac{p-1}{2}+j}{2j}\binom{2j}{j}\frac{\left(2a+b + 2\sqrt{c} \right) ^{j}\left(b-2\sqrt{c} \right) ^{j}\left(4\sqrt{c} \right) ^{\frac{p-1}{2}-j}}{(2a)^{\frac{p-1}{2}+j}}\pmod p.
\end{align*}
Finally, by Morley's congruence and Fermat's little theorem gives \eqref{5.4+}. This completes the proof.
\end{proof}
We now apply Theorem~\ref{thm5.6} to confirm some conjectures proposed by Sun \cite{Sun2014CANT}.
\begin{corollary}{\cite[Conjecture 2.3]{Sun2014CANT}} Let p be an odd prime. Then
\begin{align*}
\sum_{k=0}^{(p-1)/2}\binom{2k}{k}\dfrac{T_{2k}(12,-7)}{16^k}\equiv \begin{cases}
2x(\frac{x}{7})\pmod{p}& \text{ if } (\frac{p}{7})=1 \text{ and }p=x^2+7y^2,\\
0\pmod{p} & \text{ if }(\frac{p}{7})=-1 \text{ and }p=x^2+7y^2.
\end{cases}
\end{align*}
\end{corollary}
\begin{proof}
Taking $a=2$, $b=10$ and $c=-7$ in \eqref{5.4+} and using the fact that $P_n(-x)=(-1)^nP_n(x)$, with some simplifications, we obtain
\begin{align*}
\sum_{k=0}^{(p-1)/2}\binom{2k}{k}\dfrac{T_{2k}(12,-7)}{16^k} & \equiv (-1)^{\frac{p-1}{2}} \left( \sqrt{-7}\right) ^{\frac{p-1}{2}}P_{\frac{p-1}{2}}\left(3\sqrt{-7}\right)  \pmod p\\
&=(-3)^{\frac{p-1}{2}} \left( \sqrt{-63}\right) ^{\frac{p-1}{2}}P_{\frac{p-1}{2}}\left(\sqrt{-63}\right)  \pmod p.
\end{align*}
By \cite[Lemma 2.5]{zhsun2013}, we obtain
\begin{align*}
\sum_{k=0}^{(p-1)/2}\binom{2k}{k}\dfrac{T_{2k}(12,-7)}{16^k} & \equiv \left( \frac{-3}{p}\right)\left( \frac{-63}{p}\right)P_{[\frac{p}{4}]}\left( -\frac{65}{63}\right)    \pmod p.
\end{align*}
Sun established that (see the proof of \cite[Theorem 2.5]{zhsun2013}): for $p\neq 2,3,7$, we have
\begin{align}\label{5.5+}
\left( \frac{-63}{p}\right)P_{[\frac{p}{4}]}\left(- \frac{65}{63}\right)\equiv \begin{cases}
2x(\frac{p}{3})(\frac{x}{7})\pmod{p}& \text{ if } p\equiv 1,2,4\pmod{7} \text{ and so }p=x^2+7y^2,\\
0\pmod{p} & \text{ if }p\equiv 3,5,6\pmod{7}.
\end{cases}
\end{align}
Note that $( \frac{-3}{p})\left( \frac{p}{3}\right)=1.$ Combining this with \eqref{5.5+} yields the desired result.
\end{proof}
\begin{lemma}\label{lem:minus3quartic}
Let $p\equiv1\pmod4$ be a prime and write $p=x^2+y^2$ with $2\mid y$ and
$x+y\equiv1\pmod4$. Then
\begin{align*}
(-3)^{\frac{p-1}{4}}\equiv
\begin{cases}
-1\pmod p & \text{if }3\mid x,\\
1\pmod p & \text{if }3\mid y,\\
x/y\pmod p & \text{if }x\equiv y\pmod3,\\
-x/y\pmod p & \text{if }x\equiv-y\pmod3.
\end{cases}
\end{align*}
\end{lemma}

\begin{proof}
For an odd integer $q$ set $q^*=(-1)^{(q-1)/2}q$.
For an odd prime $p>3$ and $r\in\{0,1,2,3\}$ define
\[
Q_r(p)=\Bigl\{c\in\mathbb Z_p:\Bigl(\frac{c+i}{p}\Bigr)_4=i^r\Bigr\},
\]
where $(\frac{\cdot}{\cdot})_4$ denotes the quartic Jacobi symbol in
$\mathbb Z[i]$.  By convention $(\frac{\alpha}{1})_4=1$ for all
$\alpha\in\mathbb Z[i]$.

Apply \cite[Theorem~2.2]{Sun2001} with the auxiliary prime $q=3$.
 The hypothesis $\gcd\bigl(y,3/\gcd(y,3)\bigr)=1$ holds automatically,
because either $\gcd(y,3)=1$ or
$\gcd(y,3)=3$. Thus for $r=0,1,2,3$,
\begin{align}\label{5.7+}
(-3)^{\frac{p-1}{4}}\equiv(y/x)^r\pmod p
\quad\Longleftrightarrow\quad
x/y\in Q_r\!\Bigl(\frac{3}{\gcd(y,3)}\Bigr).
\end{align}
\textbf{Case $3\mid y$.}  Then $\gcd(y,3)=3$, so the modulus in \eqref{5.7+}
degenerates to $Q_r(1)$.  Because $(\frac{c+i}{1})_4=1=i^0$ for every
$c\in\mathbb Z_p$, we have $Q_0(1)=\mathbb Z_p$ and
$Q_1(1)=Q_2(1)=Q_3(1)=\varnothing$.  Hence $x/y\in Q_0(1)$ forces
$r=0$, and \eqref{5.7+} yields
\[
(-3)^{\frac{p-1}{4}}\equiv(y/x)^0\equiv1\pmod p.
\]

\textbf{Case $3\nmid y$.}  Then $\gcd(y,3)=1$ and we use $Q_r(3)$
directly.  A direct computation in $\mathbb F_9=\mathbb Z[i]/(3)$ shows
\[
Q_0(3)=\varnothing,\quad Q_1(3)=\{2\},\quad Q_2(3)=\{0\},\quad
Q_3(3)=\{1\}
\]
(i.e.\ $(0+i)^2\equiv i^2$, $(1+i)^2\equiv i^3$, $(2+i)^2\equiv i^1$
in $\mathbb F_9$).  Hence $x/y\bmod3\in\{0,1,2\}$ corresponds to
$r\in\{2,3,1\}$ respectively.  Since $x^2+y^2=p$, the quantity
$s:=x/y\bmod p$ satisfies $s^2\equiv-1\pmod p$; consequently
$y/x\equiv s^{-1}\equiv -s\pmod p$.  Substituting the three
possibilities into \eqref{5.7+} gives
\begin{align*}
(-3)^{\frac{p-1}{4}}\equiv
\begin{cases}
-1\pmod p & \text{if }3\mid x,\\
x/y\pmod p & \text{if }x\equiv y\pmod3,\\
-x/y\pmod p & \text{if }x\equiv-y\pmod3.
\end{cases}
\end{align*}
Combining both cases yields the four stated congruences.
\end{proof}
\begin{lemma}\label{lem:three-quartic}
Let $p\equiv1\pmod4$ be a prime and write $p=x^2+y^2$ with $2\mid y$
and $x\equiv1\pmod4$. Then
\[
3^{\frac{p-1}{4}}\equiv
\begin{cases}
(-1)^{\lfloor x/6\rfloor+y/2}\pmod p & \text{if }12\mid p-1,\\
(-1)^{(x+y+1)/2}\,y/x\pmod p & \text{if }12\mid p-5.
\end{cases}
\]
\end{lemma}
\begin{proof}
Since $x$ is odd, $x^2\equiv1\pmod8$. Thus, one observes that
 $4\mid y\iff p\equiv1\pmod8$ and
$y\equiv2\pmod4\iff p\equiv5\pmod8$. We distinguish two cases.
\smallskip

\textbf{Case $p\equiv1\pmod8$.} Since
$x\equiv1\pmod4$ we get $x+y\equiv1\pmod4$; that is, $(x,y)$ already
satisfies the hypothesis of Lemma~\ref{lem:minus3quartic}. Also $(p-1)/4$ is even, so
$(-1)^{(p-1)/4}=(-1)^{y/2}=1$, and $3^{(p-1)/4}\equiv(-3)^{(p-1)/4}
\pmod p$.
\smallskip

\textbf{Case $p\equiv5\pmod8$.} Here $y\equiv2\pmod4$, so
$x+y\equiv3\pmod4$; thus $(x,y)$ does \emph{not} satisfy the
hypothesis of Lemma~\ref{lem:minus3quartic}, but $(-x,y)$ does, since
$-x+y\equiv1\pmod4$. Also, $(p-1)/4$ is odd here, so $3^{(p-1)/4}\equiv-(-3)^{(p-1)/4}\pmod p$.

Combining both cases, we deduce
\begin{equation}\label{eq:pf-uniform}
3^{\frac{p-1}{4}}\equiv
\begin{cases}
-(-1)^{y/2}\pmod p & 3\mid x,\\
(-1)^{y/2}\pmod p & 3\mid y,\\
(-1)^{y/2}x/y\pmod p & x\equiv y\pmod3,\\
-(-1)^{y/2} x/y \pmod p & x\equiv-y\pmod3.
\end{cases}
\end{equation}
Now suppose $3\mid x$. Since $x\equiv1\pmod4$, we have $x\equiv9\pmod{12}$, and hence $\lfloor x/6\rfloor$ is odd. If
$3\mid y$, then $3\nmid x$, so $x\equiv1,5\pmod{12}$, for which
$\lfloor x/6\rfloor$ is even. In either sub-case,
\eqref{eq:pf-uniform} becomes
\[
3^{\frac{p-1}{4}}\equiv(-1)^{\lfloor x/6\rfloor+y/2}\pmod p,
\]
proving the case $12\mid p-1$.

\smallskip
\textbf{The case $12\mid p-5$.} Here $3\nmid xy$. Observe that 
\begin{equation}\label{eq:pf-recip}
y/x\equiv-x/y\pmod p.
\end{equation}
Define $T(\beta):=(-1)^{(x+\beta+1)/2}\,\beta/x$ for any integer
$\beta$. Note that for $\beta=y$,
\begin{equation}\label{eq:T-flip}
T(-y)=-T(y).
\end{equation}
Suppose that $x\equiv y\pmod3$. Since $x\equiv 1 \pmod 4$, and by \eqref{eq:pf-uniform} and \eqref{eq:pf-recip}, we obtain
$$3^{(p-1)/4}\equiv T(y)\pmod p. $$
Now, suppose that $x\equiv-y\pmod3$. It is easy to show that $(x-1)/2\equiv 0\pmod 4$. This, together with \eqref{eq:pf-uniform} and \eqref{eq:T-flip} yields
$$3^{(p-1)/4}\equiv T(-y)\pmod p.$$
Combining the two sub-cases gives 
\[
3^{\frac{p-1}{4}}\equiv T(\pm y)=(-1)^{\frac{x+y+1}{2}}y/x\pmod p,
\]
proving the case $12\mid p-5$ and completing the proof.
\end{proof}
\begin{corollary}{\cite[Conjecture 2.2, Part 3]{Sun2014CANT}} Let $p > 3$ be a prime. Then
\begin{align*}
\sum_{k=0}^{(p-1)/2}\binom{2k}{k}\dfrac{T_{2k}(4,3)}{16^k}\equiv
\begin{cases}
2x(-1)^{\lfloor x/6\rfloor+y/2}\pmod p &\text{ if } 12\mid p-1 \text{ and }p=x^2+y^2 \text{ }(2\nmid x),\\
2y(-1)^{(x+y+1)/2}\pmod p &\text{ if }12\mid p-5 \text{ and }p=x^2+y^2 \text{ }(2\nmid x),\\
0\pmod{p}& \text{ if }p\equiv 3 \pmod4.
\end{cases}
\end{align*}
\end{corollary}
\begin{proof}
Taking $a=2$, $b=2$ and $c=3$ in \eqref{5.4+}, we obtain
\begin{align*}
\sum_{k=0}^{(p-1)/2}\binom{2k}{k}\dfrac{T_{2k}(4,3)}{16^k}\equiv \left( \sqrt{3}\right) ^{\frac{p-1}{2}}P_{\frac{p-1}{2}}(0)\pmod{p}.
\end{align*}
It is known that 
\begin{align*}
P_{n}(0)=\begin{cases}
0 &\text{ if } n\text{ is odd},\\
\binom{2m}{m}\frac{(-1)^m}{4^m} & \text{ if } n=2m.
\end{cases}
\end{align*}
Hence, if $p\equiv 3\pmod 4$, we deduce 
\begin{align*}
\sum_{k=0}^{(p-1)/2}\binom{2k}{k}\dfrac{T_{2k}(4,3)}{16^k}\equiv 0\pmod p.
\end{align*}
 Gauss's congruence states that for a prime \(p \equiv 1 \pmod 4\) written uniquely as \(p = x^2 + y^2\) with \(x \equiv 1 \pmod 4\),
 \begin{align*}
\binom{\frac{p-1}{2}}{\frac{p-1}{4}}\equiv 2x\pmod p.
 \end{align*}
It follows that
\begin{align*}
\sum_{k=0}^{(p-1)/2}\binom{2k}{k}\dfrac{T_{2k}(4,3)}{16^k}&\equiv 2x \left(\frac{2}{p}\right) (-3)^{\frac{p-1}{4}}\pmod p\\
&\equiv (-1)^{\frac{p^2-1}{8}}2x(-3)^{\frac{p-1}{4}}\pmod p\\
&=2x3^{\frac{p-1}{4}}\pmod p.
\end{align*}
Applying Lemma \ref{lem:three-quartic} and combining the result with the above case yields the desired result. 
\end{proof}

\begin{lemma}\label{lem:minus2}
Let $p\equiv1\pmod3$ be a prime and write $p=x^2+3y^2$ with $3\mid x-1$.
Then
\begin{align*}
\left(\frac{-2}{p}\right)=(-1)^{xy/2}\left(\frac x3\right).
\end{align*}
\end{lemma}

\begin{proof}
Since $3\mid x-1$, we have $\big(\frac x3\big)=1$, so it suffices to show
$\big(\frac{-2}p\big)=(-1)^{xy/2}$. As $p$ is odd, exactly one of $x,y$ is
even; write $2c$ for whichever of $x,y$ is even, so that $xy/2\equiv c
\pmod2$. A direct computation modulo $8$ (using $x^2\equiv1$ or $0,4\pmod8$
according to parity, and likewise for $3y^2$) shows
\[
p=x^2+3y^2\equiv\begin{cases}4c^2+3\pmod8& \text{ if }x=2c,\\
1+12c^2\pmod8 &\text{ if }y=2c.\end{cases}
\]
Hence $p\equiv3,1\pmod8$ if $c$ is even and $p\equiv7,5\pmod8$ if $c$ is odd,
respectively. Since $\big(\frac{-2}p\big)=1$ for $p\equiv1,3\pmod8$ and
$\big(\frac{-2}p\big)=-1$ for $p\equiv5,7\pmod8$, in every case
$\big(\frac{-2}p\big)=(-1)^c=(-1)^{xy/2}$, as required.
\end{proof}
\begin{corollary}{\cite[Conjecture 2.2, Part 2]{Sun2014CANT}} Let $p > 3$ be a prime. Then
\begin{align*}
\sum_{k=0}^{(p-1)/2}\binom{2k}{k}\dfrac{T_{2k}(1,-3)}{4^k}\equiv \begin{cases}
(-1)^{xy/2}(\frac{x}{3})2x\pmod{p}& \text{ if } p=x^2+3y^2,\\
0\pmod{p} & \text{ if } p\equiv 2\pmod 3 .
\end{cases}
\end{align*}
\end{corollary}
\begin{proof}
Taking $a=1$, $b=0$, and $c=-3$ in \eqref{5.4+} and by the fact that $P_n(-x)=(-1)^nP_n(x)$, we have
\begin{align*}
\sum_{k=0}^{(p-1)/2}\binom{2k}{k}\dfrac{T_{2k}(1,-3)}{4^k}\equiv\left( \frac{-2}{p}\right) \left( \sqrt{-3}\right) ^{\frac{p-1}{2}}P_{\frac{p-1}{2}}\left( \sqrt{-3}\right)\pmod p.
\end{align*}
In view of \cite[Theorem 2.1(i)]{zhsun2013}, we obtain
\begin{align*}
\sum_{k=0}^{(p-1)/2}\binom{2k}{k}\dfrac{T_{2k}(1,-3)}{4^k}\equiv\left( \frac{2}{p}\right)\left( \frac{3}{p}\right) P_{[\frac{p}{4}]}\left( -\frac{5}{3}\right)\pmod p.
\end{align*}
Hence, by \cite[Theorem 2.3]{zhsun2013}, we have
\begin{align*}
\sum_{k=0}^{(p-1)/2}\binom{2k}{k}\dfrac{T_{2k}(1,-3)}{4^k}\equiv\begin{cases}
2x\left( \frac{-2}{p}\right) \pmod p &\text{ if }3\mid p-1\text{, }p=x^2+3y^2\text{ and }3\mid x-1,\\
0 \pmod p & \text{ if }p\equiv 2\pmod 3.
\end{cases}
\end{align*}
Finally, Lemma \ref{lem:minus2} gives the desired result.
\end{proof}
\begin{corollary}Let $p > 3$ be a prime. Then \eqref{suncnj2.2} is true modulo $p$.
\end{corollary}
\begin{proof}
Substituting $a=2$, $b=0$, $c=3$ into \eqref{5.4+}, we have
\begin{align*}
\sum_{k=0}^{(p-1)/2}\binom{2k}{k}\dfrac{T_{2k}(2,3)}{16^k}\equiv\left(\frac{-1}{p}\right)\left(\sqrt{3}\right)^{\frac{p-1}{2}}P_{\frac{p-1}{2}}\left(\frac{\sqrt{3}}{2}\right)\pmod{p}.
\end{align*}
By \cite[Lemma 2.5]{zhsun2013}, we obtain
\begin{align}\label{5.6++}
\sum_{k=0}^{(p-1)/2}\binom{2k}{k}\dfrac{T_{2k}(2,3)}{16^k}&\equiv \left(\frac{-1}{p}\right)\left(\sqrt{-3}\right)^{\frac{p-1}{2}}P_{\frac{p-1}{2}}\left(\sqrt{-3}\right)\pmod{p}\\
&\equiv\left(\frac{-1}{p}\right)\left(\frac{-3}{p}\right)P_{[\frac{p}{4}]}\left(-\frac{5}{3}\right)\pmod{p}\nonumber .
\end{align}
Thus, by \cite[Theorem 2.3]{zhsun2013}, 
\begin{align*}
\sum_{k=0}^{(p-1)/2}\binom{2k}{k}\dfrac{T_{2k}(2,3)}{16^k}\equiv \begin{cases}
2x\left( \frac{-1}{p}\right) \pmod p &\text{ if }3\mid p-1\text{, }p=x^2+3y^2\text{ and }3\mid x-1,\\
0 \pmod p & \text{ if }p\equiv 2\pmod 3.
\end{cases}
\end{align*}
On the other hand, taking $a=2$, $b=2$, $c=-3$ in \eqref{5.4+}, we obtain
\begin{align*}
\sum_{k=0}^{(p-1)/2}\binom{2k}{k}\dfrac{T_{2k}(4,-3)}{16^k}\equiv \left(\frac{-1}{p}\right)\left(\sqrt{-3}\right)^{\frac{p-1}{2}}P_{\frac{p-1}{2}}\left(\sqrt{-3}\right)\pmod{p}.
\end{align*}
In view of \eqref{5.6++}, we have
\begin{align*}
\sum_{k=0}^{(p-1)/2}\binom{2k}{k}\dfrac{T_{2k}(4,-3)}{16^k}\equiv\sum_{k=0}^{(p-1)/2}\binom{2k}{k}\dfrac{T_{2k}(2,3)}{16^k}\pmod p.
\end{align*}
Hence the proof is done.
\end{proof}
\vspace{1cm}
\textbf{Acknowledgments.} This paper is partially supported by DGRSDT grant $No^{\circ}$ C0656701.

\end{document}